\documentclass[12pt,a4paper]{amsart}

\usepackage[english]{babel}
\usepackage[utf8]{inputenc}
\usepackage[T1]{fontenc}
\usepackage{latexsym}
\usepackage{amsmath,amssymb,amsthm,amsfonts}
\usepackage{mathtools}
\usepackage{txfonts}						
\usepackage{stmaryrd}
\usepackage{physics}
\usepackage{relsize}
\usepackage{framed, color}
\usepackage{hyperref}
\usepackage{dsfont}
\usepackage{cancel}
\usepackage[normalem]{ulem}
\usepackage{xcolor}
\usepackage{nicefrac}

\allowdisplaybreaks

\usepackage[left=2cm, right=2cm, bottom=3cm]{geometry}
\theoremstyle{plain}
\newtheorem{thm}{Theorem}[section]
\newtheorem{defn}[thm]{Definition}
\newtheorem{prop}[thm]{Proposition} 
\newtheorem{lemma}[thm]{Lemma}

\newtheorem{rmk}[thm]{Remark}\theoremstyle{remark} 
                  
\newcommand{\e}{\mathrm{e}} 
\renewcommand{\i}{\mathrm{i}} 
\newcommand{\C}{\mathbb{C}} 
\newcommand{\R}{\mathbb{R}} 
\newcommand{\Z}{\mathbb{Z}} 
\newcommand{\N}{\mathbb{N}} 
\newcommand{\T}{\mathbb{T}} 
\newcommand{\skp}[2]{\left< #1 ,#2 \right>} 

\newcommand{\difff}[2]{\frac{\mathrm{d}^{#2}}{\mathrm{d}#1^{#2}}} 
\newcommand{\Dom}[1]{{\mathcal{D}(#1)}} 
\newcommand{\SobH}[2]{{H^{#1}(#2)}} 
\newcommand{\Leb}[2]{{L^{#1}(#2)}} 
\newcommand{\cont}[1]{{C^{#1}}} 
\newcommand{\Cont}[2]{{C^{#1}(#2)}} 
\newcommand{\Contc}[2]{{C^{#1}_c(#2)}} 
\newcommand{\seq}[1]{{l^{#1}}} 
\newcommand{\seqsobh}[1]{{h^{#1}}} 

\newcommand{\Zodd}{{\Z_{odd}}}
\newcommand{\Nodd}{{\N_{odd}}}
\newcommand{\Neven}{{\N_{even}}}

\newcommand{\SobHper}[2]{{H^{#1}_{per}(#2)}} 

\newcommand{\RT}{{\R\times\T_T}}
\DeclareMathOperator{\dist}{dist}

\newcommand{\NORM}[1]{{\interleave #1\interleave}}

\begin{document}

\title{Breather Solutions for a Quasilinear $(1+1)$-dimensional Wave Equation}

\author{Simon Kohler and Wolfgang Reichel}
\address{Institute for Analysis, Karlsruhe Institute of Technology (KIT), D-76128 Karlsruhe, Germany}
\email{simon.kohler@kit.edu, wolfgang.reichel@kit.edu}

\date{\today}

\begin{abstract}
	We consider the $(1+1)$-dimensional quasilinear wave equation $g(x)w_{tt}-w_{xx}+h(x) (w_t^3)_t=0$ on $\R\times\R$ which arises in the study of localized electromagnetic waves modeled by Kerr-nonlinear Maxwell equations. We are interested in time-periodic, spatially localized solutions. Here $g\in\Leb{\infty}{\R}$ is even with $g\not\equiv 0$ and $h(x)=\gamma\,\delta_0(x)$ with $\gamma\in\R\backslash\{0\}$ and $\delta_0$ the delta-distribution supported in $0$. We assume that $0$ lies in a spectral gap of the operators $L_k=-\difff{x}{2}-k^2\omega^2g$ on $\Leb{2}{\R}$ for all $k\in 2\Z+1$ together with additional properties of the fundamental set of solutions of $L_k$. By expanding $w$ into a Fourier series in time we transfer the problem of finding a suitably defined weak solution to finding a minimizer of a functional on a sequence space.  The solutions that we have found are exponentially localized in space. Moreover, we show that they can be well approximated by truncating the Fourier series in time. The guiding examples, where all assumptions are fulfilled, are explicitely given step potentials and periodic step potentials $g$. In these examples we even find infinitely many distinct breathers. 
\end{abstract}

\maketitle

\section{Introduction and Main Results}\label{results}
	We study the $(1+1)$-dimensional quasilinear wave equation
	\begin{align}
		g(x)w_{tt}-w_{xx}+h(x)(w_t^3)_t=0 \quad \text{for } (x,t)\in\R\times\R,
		\label{quasi}
	\end{align}
	and we look for real-valued, time-periodic and spatially localized solutions $w(x,t)$. At the end of this introduction we give a motivation for this equation arising in the study of localized electromagnetic waves modeled by Kerr-nonlinear Maxwell equations. We also cite some relevant papers. To the best of our knowledge for \eqref{quasi} in its general form no rigorous existence results are available. A first result is given in this paper by taking an extreme case where $h(x)$ is a spatial delta distribution  at $x=0$. Our basic assumption on the coefficients $g$ and $h$ is the following: 
	\begin{align}
		g\in\Leb{\infty}{\R} \text{ even, } g \not \equiv 0 \text{ and } h(x)=\gamma\delta_0(x) \mbox{ with } \gamma\neq0  \tag{C0}\label{C0}
	\end{align}
	where $\delta_0$ denotes the delta-distribution supported in $0$. We have two prototypical examples for the potential $g$: a step potential (Theorem~\ref{step}) and a periodic step potential (Theorem~\ref{w is a weak solution in expl exa}). The general version is given in Theorem~\ref{w is a weak solution general} below.
	
	\begin{thm}\label{step}
	For $a,b,c>0$ let 
		\[
		g(x)\coloneqq\left\{\begin{array}{ll}
		-a, & \mbox{ if }\, \abs{x}>c,\\
		b, & \mbox{ if }\, \abs{x}<c.
		\end{array}\right.
		\]
        For every frequency $\omega$ such that $\sqrt{b}\omega c \frac{2}{\pi}\in \frac{2\N+1}{2\N+1}$ and $\gamma<0$ there exist infinitely many nontrivial, real-valued, spatially localized and time-periodic weak solutions of \eqref{quasi} with period $T=\frac{2\pi}{\omega}$. For each solution $w$ there are constants $C,\rho>0$ such that $\abs{w(x,t)}\leq C\e^{-\rho\abs{x}}$.
	\end{thm}
	
	\begin{thm}\label{w is a weak solution in expl exa}
		For $a,b>0$, $a\not =b$ and $\Theta\in(0,1)$ let 
		\[
		g(x)\coloneqq\left\{\begin{array}{ll}
		a, & \mbox{ if }\, \abs{x}<\pi\Theta,\\
		b, & \mbox{ if }\, \pi\Theta<\abs{x}<\pi
		\end{array}\right.
		\]
		and extend $g$ as a $2\pi$-periodic function to $\R$.
		Assume in addition 
		\begin{align*}
			\sqrt{\frac{b}{a}}\,\frac{1-\Theta}{\Theta}\in  \frac{2\N+1}{2\N+1}. 
		\end{align*}
		For every frequency $\omega$ such that $4\sqrt{a}\theta\omega\in \frac{2\N+1}{2\N+1}$ there exist infinitely many nontrivial, real-valued, spatially localized and time-periodic weak solutions of \eqref{quasi} with period $T=\frac{2\pi}{\omega}$. For each solution $w$ there are constants $C,\rho>0$ such that $\abs{w(x,t)}\leq C\e^{-\rho\abs{x}}$.
	\end{thm}

	Weak solutions of \eqref{quasi} are understood in the following sense. We write $D\coloneqq\RT$ and denote by $\T_T$ the one-dimensional torus with period $T$.
	
	\begin{defn}\label{Defn of weak Sol to (quasi)}
		Under the assumption \eqref{C0} a function $w\in\SobH{1}{\RT}$ with $\partial_tw(0,\cdot)\in\Leb{3}{\T_T}$ is called a weak solution of \eqref{quasi} if for every $\psi\in\Contc{\infty}{\RT}$  
		\begin{align}
			\int_D-g(x)\partial_tw\,\partial_t\psi +\partial_xw\,\partial_x\psi\dd{(x,t)} -\gamma\int_{0}^{T} (\partial_t w(0,t))^3 \partial_t \psi(0,t)\dd{t}=0. \label{WeakEquation for (quasi)}
		\end{align}
	\end{defn}	
	Theorem~\ref{step} and Theorem~\ref{w is a weak solution in expl exa} are special cases of Theorem~\ref{w is a weak solution general}, which applies to much more general potentials $g$. In Section~\ref{details_example_step} and Section~\ref{explicit example Bloch Modes_WR} of the Appendix  we will show that the special potentials $g$ from these two theorems satisfy the conditions \eqref{spectralcond} and \eqref{FurtherCond_phik} of Theorem~\ref{w is a weak solution general}. The basic preparations and assumptions for Theorem~\ref{w is a weak solution general} will be given next.
	
	Since we are looking for time-periodic solutions, it is appropriate to make the Fourier ansatz $w(x,t)=\sum_{k\in\Zodd} w_k(x) \e^{\i k\omega t}$ with $\Zodd\coloneqq2\Z+1$. The reason for dropping even Fourier modes is that the $0$-mode belongs to the kernel of the wave operator $L=g(x)\partial_t^2 - \partial_x^2$. The restriction to odd Fourier modes generates $T/2=\pi/\omega$-antiperiodic functions $w$, is therefore compatible with the structure of \eqref{quasi} and in particular the cubic nonlinearity. Notice the decomposition $(Lw)(x,t)=\sum_{k\in \Zodd} L_k w_k(x) \e^{\i k\omega t}$ with self-adjoint operators 
	\begin{align*}
		L_k = -\frac{d^2}{dx^2} - k^2\omega^2 g(x): H^2(\R)\subset L^2(\R)\rightarrow L^2(\R).
	\end{align*}
	Clearly $L_k=L_{-k}$ so that is suffices to study $L_k$ for $k\in \Nodd$. Our first assumption is concerned with the spectrum $\sigma(L_k)$:
	\begin{align}\label{spectralcond}
		\forall\,k\in \Nodd, 0\not \in \sigma_{\mathrm{ess}}(L_k)\cup \sigma_{\mathrm{D}}(L_k), \tag{C1}
	\end{align}
	where by $\sigma_{\mathrm{D}}(L_k)$ we denote the spectrum of $L_k$ with an extra Dirichlet condition at $0$, i.e., the spectrum of $L_k$ restricted to $\{\varphi\in\SobH{2}{\R}~|~\varphi(0)=0\}$. This is the same as the spectrum of $L_k$ restricted to functions which are odd around $x=0$.
	
	\begin{lemma} \label{exp_decaying_sol} 
		Under the assumption \eqref{C0} and \eqref{spectralcond} there exists for every $k\in \Nodd$ a function $\Phi_k\in H^2(0,\infty)$ with $L_k\Phi_k=0$ on $(0,\infty)$ and $\Phi_k(0)=1$. 
	\end{lemma}
	\begin{proof} 
		We have either that $0$ is in the point spectrum (but not the Dirichlet spectrum) or that $0$ is in the resolvent set of $L_k$. In the first case there is an eigenfunction $\Phi_k\in H^2(\R)$ with $L_k\Phi_k=0$ and w.l.o.g. $\Phi_k(0)=1$. In the second case $0\in \rho(L_k)$ so that there exists a unique solution $\tilde \Phi_k$ of $L_k \tilde \Phi_k = 1_{[-2,-1]}$ on $\R$. Clearly, if restricted to $(0,\infty)$, the function $\tilde\Phi_k$ solves $L_k \tilde\Phi_k=$ on $(0,\infty)$. Moreover, $\tilde\Phi_k(0)\not =0$ since otherwise $\tilde\Phi_k$ would be an odd eigenfunction of $L_k$ which is excluded due to $0\in \rho(L_k)$. Thus a suitably rescaled version of $\tilde\Phi_k$ satisfies the claim of the lemma.
	\end{proof}
	
	Our second set of assumptions concerns the structure of the decaying fundamental solution according to Lemma~\ref{exp_decaying_sol}.
	\begin{align}\label{FurtherCond_phik}
		\mbox{There exist $\rho, M>0$ such that for all } k\in \Nodd\colon |\Phi_k(x)|\leq Me^{-\rho 	x} \mbox{ on } [0,\infty). \tag{C2}
	\end{align}
	
	Now we can formulate our third main theorem as a generalization of Theorem~\ref{step} and Theorem~\ref{w is a weak solution in expl exa}. The fact that the solutions which we find, can be well approximated by truncation of the Fourier series in time, is explained in Lemma~\ref{lemma_approximation} below. Moreover, a further extension yielding infinitely many different solutions is given in Theorem~\ref{multiplicity abstract} in Section~\ref{infinitely_many_breathers}.
	
	\begin{thm}\label{w is a weak solution general}
		Assume \eqref{C0}, \eqref{spectralcond} and \eqref{FurtherCond_phik} for a potential $g$ and a frequency $\omega>0$. Then \eqref{quasi} has a nontrivial, $T$-periodic weak solution $w$ in the sense of Definition~\ref{Defn of weak Sol to (quasi)} with $T=\frac{2\pi}{\omega}$ provided 
		\begin{itemize}
		\item[(i)] $\gamma<0$ and the sequence $\left(\Phi'_k(0)\right)_{k\in\Nodd}$ has at least one positive element, 
		\item[(ii)] $\gamma>0$ and the sequence $\left(\Phi'_k(0)\right)_{k\in\Nodd}$ has at least one negative element. 
		\end{itemize}
		Moreover, there is a constant $C>0$ such that $\abs{w(x,t)}\leq C\e^{-\rho\abs{x}}$ for all $(x,t)\in \R^2$ with $\rho$ as in \eqref{FurtherCond_phik}.
	\end{thm}
	
	\begin{rmk} 
		(a) \label{remark_Dr} It turns out that the above assumptions can be weakened as follows: it suffices to verify \eqref{spectralcond} and \eqref{FurtherCond_phik} and (i), (ii) for all integers $k\in r\cdot \Zodd$ for some $r\in \Nodd$. We will prove this observation in Section~\ref{infinitely_many_breathers}.

(b) Our variational approach also works if we consider \eqref{quasi} with Dirichlet boundary conditions on a bounded interval $(-l,l)$ instead of the real line. There are many possible results. For illustration purposes we just formulate the simplest one. E.g., if we assume that $\frac{\omega l}{\pi}\in\frac{\Zodd}{4\Z}$ then 
			\begin{align*}
				w_{tt}-w_{xx}+\gamma\delta_0(x)(w_t^3)_t=0 \mbox{ on } (-l,l)\times\R \mbox{ with } w(\pm l,t)=0 \mbox{ for all } t
			\end{align*}
has a nontrivial, real-valued time-periodic weak solution with period $T=\frac{2\pi}{\omega}$ both for $\gamma>0$ and $\gamma<0$. The operator $L_k=-\frac{d^2}{dx^2}-\omega^2k^2$ is now a self-adjoint operator on $H^2(-l,l)\cap H_0^1(-l,l)$. The assumption $\frac{\omega l}{\pi}\in\frac{\Zodd}{4\Z}$ guarantees \eqref{spectralcond} for all $k\in\Zodd$. The functions $\Phi_k$ are given by $\Phi_k(x)=\frac{\sin(\omega k(l-x))}{\sin(\omega kk)}$ so that $\Phi_k'(0)=-\omega k\cot(\omega kl)$. The assumption $\frac{\omega l}{\pi}\in\frac{\Zodd}{4\Z}$ now guarantees that the sequence $\{\cot(\omega kl)~|~k\in\Zodd\}$ is finite and does not contain $0$ or $\pm\infty$. Moreover $\frac{\omega l}{\pi}=\frac{2p+1}{4q}$ yields $\Phi_{k}'(0)\Phi_{k+2q}'(0)<0$, i.e., we also have the required sign-change which allows for both signs of $\gamma$.   
\end{rmk}
	
	We observe that the growth of $\left(\Phi'_k(0)\right)_{k\in\Zodd}$ is connected to regularity properties of our solutions.
	\begin{thm}\label{w is even more regular}
		Assume \eqref{C0}, \eqref{spectralcond} and \eqref{FurtherCond_phik} and additionally $\Phi'_k(0) = O(k)$. Then the weak solution $w$ from Theorem \ref{w is a weak solution general} belongs to $\SobHper{1+\nu}{\T_T,\Leb{2}{\R}}\cap\SobHper{\nu}{\T_T,\SobH{1}{\R}}$ for any $\nu\in(0,\frac{1}{4})$.
	\end{thm}
	
	Here, for $\nu\in\R$ the fractional Sobolev spaces of time-periodic functions are defined by
	\begin{align*}
		\SobHper{\nu}{\T_T,\Leb{2}{\R}}&\coloneqq\left\{ u(x,t)=\sum_{k\in\Z}\hat{u}_k(x)\e^{\i\omega kt} ~\bigg|~ \sum_{k\in\Z}\left(1+\abs{k}^2\right)^{\nu}\norm{\hat{u}_k}^2_\Leb{2}{\R}<\infty \right\}, \\
		\SobHper{\nu}{\T_T,\SobH{1}{\R}}&\coloneqq\left\{ u(x,t)=\sum_{k\in\Z}\hat{u}_k(x)\e^{\i\omega kt} ~\bigg|~ \sum_{k\in\Z}\left(1+\abs{k}^2\right)^{\nu}\norm{\hat{u}_k}^2_\SobH{1}{\R}<\infty \right\}.
	\end{align*}

	We shortly motivate \eqref{quasi} and give some references to the literature. Consider Maxwell's equations in the absence of charges and currents
	\begin{align*}
		\nabla\cdot\mathbf{D}&=0, &\nabla\times\mathbf{E}\,=&-\partial_t\mathbf{B}, &\mathbf{D}=&\varepsilon_0\mathbf{E}+\mathbf{P}(\mathbf{E}), \\
		\nabla\cdot\mathbf{B}&=0, &\nabla\times\mathbf{H}=&\,\partial_t\mathbf{D}, &\mathbf{B}=&\mu_0\mathbf{H}. 
	\end{align*}
	We assume that the dependence of the polarization $\mathbf{P}$ on the electric field $\mathbf{E}$ is instantaneous and it is the sum of a linear and a cubic term given by $\mathbf{P}(\mathbf{E})=\varepsilon_0\chi_1(\mathbf{x})\mathbf{E}+\varepsilon_0\chi_3(\mathbf{x})\abs{\mathbf{E}}^2\mathbf{E}$, cf. \cite{agrawal}, Section~2.3 (for simplicity, more general cases where instead of a factor multiplying $\abs{\mathbf{E}}^2\mathbf{E}$ one can take $\chi_3$ as an $\mathbf{x}$-dependent tensor of type $(1,3)$ are not considered here). Here $\varepsilon_0, \mu_0$ are constants such that $c^2=(\varepsilon_0\mu_0)^{-1}$ with $c$ being the speed of light in vacuum and $\chi_1, \chi_3$ are given material functions. By direct calculations one obtains the quasilinear curl-curl-equation
	\begin{align}
		0=\nabla\times\nabla\times\mathbf{E} +\partial_t^2\left( V(\mathbf{x})\mathbf{E}+\Gamma(\mathbf{x})\abs{\mathbf{E}}^2\mathbf{E}\right),
		\label{curlcurl}
	\end{align}
	where $V(\mathbf{x})=\mu_0\varepsilon_0\left(1+\chi_1(\mathbf{x})\right)$ and $\Gamma(\mathbf{x})=\mu_0\varepsilon_0\chi_3(\mathbf{x})$. Once \eqref{curlcurl} is solved for the electric field $\mathbf{E}$, the magnetic induction $\mathbf{B}$ is obtained by time-integration from $\nabla\times\mathbf{E}=-\partial_t\mathbf{B}$ and it will satisfy $\nabla\cdot\mathbf{B}=0$ provided it does so at time $t=0$. By construction, the magnetic field $\mathbf{H}=\frac{1}{\mu_0} \mathbf{B}$ satisfies $\nabla\times\mathbf{H}=\partial_t\mathbf{D}$. In order to complete the full set of nonlinear Maxwell's equations one only needs to check Gauss's law $\nabla\cdot\mathbf{D}=0$ in the absence of external charges. This will follow directly from the constitutive equation $\mathbf{D}=\varepsilon_0(1+\chi_1(\mathbf{x}))\mathbf{E}+\varepsilon_0\chi_3(\mathbf{x})\abs{\mathbf{E}}^2\mathbf{E}$ and the two different specific forms of $\mathbf{E}$ given next: 
	\begin{align*}
		\mathbf{E}(\mathbf{x},t)&=(0,u(x_1-\kappa t,x_3),0)^T &&\hspace*{-2cm}\mbox{ polarized wave traveling in $x_1$-direction }  \\
		\mathbf{E}(\mathbf{x},t)&=(0,u(x_1,t),0)^T &&\hspace*{-2cm}\mbox{ polarized standing wave} 
	\end{align*}
	In the first case $\mathbf{E}$ is a polarized wave independent of $x_2$ traveling with speed $\kappa$ in the $x_1$ direction and with profile $u$. If additionally $V(\mathbf{x})=V(x_3)$ and $\Gamma(\mathbf{x})=\Gamma(x_3)$ then the quasilinear curl-curl-equation \eqref{curlcurl} turns into the following equation for $u=u(\tau,x_3)$ with the moving coordinate $\tau=x_1-\kappa t$: 
	\begin{align*}
		-u_{x_3 x_3} + (\kappa^2 V(x_3)-1) u_{\tau\tau} + \kappa^2\Gamma(x_3)(u^3)_{\tau\tau}=0.
	\end{align*}
	Setting $u=w_\tau$ and integrating once w.r.t. $\tau$ we obtain \eqref{quasi}. 
\medskip

	In the second case $\mathbf{E}$ is a polarized standing wave which is independent of $x_2, x_3$. If we assume furthermore that $V(\mathbf{x})=V(x_1)$ and $\Gamma(\mathbf{x})=\Gamma(x_1)$ then this time the quasilinear curl-curl-equation \eqref{curlcurl} for $u=w_t$ turns (after one time-integration) directly into \eqref{quasi}.
\medskip

	In the literature, \eqref{curlcurl} has mostly been studied by considering time-harmonic waves $\mathbf{E}(\mathbf{x},t)= \mathbf{U}(\mathbf{x})e^{\i\kappa t}$. This reduces the problem to the stationary elliptic equation 
	\begin{equation} \label{curl_curl_stat}
		0=\nabla\times\nabla\times\mathbf{U} -\kappa^2\left( 	V(\mathbf{x})\mathbf{U}+\Gamma(\mathbf{x})\abs{\mathbf{U}}^2\mathbf{U}\right) \mbox{ in } \R^3.
	\end{equation}
	Here case $\mathbf{E}$ is no longer real-valued. This may be justified by extending the ansatz to 
	$\mathbf{E}(\mathbf{x},t)= \mathbf{U}(\mathbf{x})e^{\i\kappa t}+c.c.$ and by either neglecting higher harmonics generated from the cubic nonlinearity or by assuming the time-averaged constitutive relation 
	$\mathbf{P}(\mathbf{E})=\varepsilon_0\chi_1(\mathbf{x})\mathbf{E}+\varepsilon_0\chi_3(\mathbf{x})\frac{1}{T}\int_0^T\abs{\mathbf{E}}^2\,dt \mathbf{E}$ with $T=2\pi/\kappa$, cf. \cite{stuart_1993}, \cite{Sutherland03}. For results on \eqref{curl_curl_stat} we refer to \cite{BDPR_2016}, \cite{Mederski_2015} and in particular to the survey \cite{Bartsch_Mederski_survey}. Time-harmonic traveling waves have been found in a series of papers \cite{stuart_1990, stuart_1993,stuart_zhou_2010}. The number of results for monochromatic standing polarized wave profiles $U(\mathbf{x})=(0,u(x_1),0)$ with $u$ satisfying $0=-u''-\kappa^2\left( V(x_1)u+\Gamma(x_1)|u|^2u\right)$ on $\R$ is too large to cite so we restrict ourselves to Cazenave's book \cite{cazenave}. 
\medskip

	Our approach differs substantially from the approaches by monochromatic waves described above. Our ansatz $w(x,t)=\sum_{k\in\Zodd} w_k(x) \e^{\i k\omega t}$ with $\Zodd\coloneqq2\Z+1$ is automatically polychromatic since it couples all integer multiples of the frequency $\omega$. A similar polychromatic approach is considered in
    \cite{PelSimWeinstein}. The authors seek spatially localized traveling wave solutions of the 1+1-dimensional quasilinear Maxwell model, where in the direction of propagation $\chi_1$ is a periodic arrangement of delta functions. Based on a multiple scale approximation ansatz, the field profile is expanded into infinitely many modes which are time-periodic in both the fast and slow time variables. Since the periodicities in the fast and slow time variables differ, the field becomes quasiperiodic in time. To a certain extent the authors of \cite{PelSimWeinstein} analytically deal with the resulting system for these infinitely many coupled modes through bifurcation methods, with a rigorous existence proof still missing.  However, numerical results from \cite{PelSimWeinstein} indicate that spatially localized traveling waves could exist. 
\medskip

    With our case of allowing $\chi_1$ to be a bounded function but taking $\chi_3$ to be a delta function at $x=0$ we consider an extreme case. On the other hand our existence results (possibly for the first time) rigorously establish localized solutions of the full nonlinear Maxwell problem \eqref{curlcurl} without making the assumption of either neglecting higher harmonics or of assuming a time-averaged nonlinear constitutive law. 
\medskip

    The existence of localized breathers of the quasilinear problem \eqref{quasi} with bounded coefficients $g, h$ remains generally open. We can, however, provide specific functions $g$, $h$ for which \eqref{quasi} has a breather-type solution that decays to $0$ as $|x|\to \infty$. Let 
	\begin{align*}
	    b(x) \coloneqq (1+x^2)^{-1/2}, \quad h(x) \coloneqq \frac{1-2x^2}{1+x^2}, \quad g(x) \coloneqq 	\frac{2+x^4}{(1+x^2)^2}
	\end{align*}
	and consider a time-periodic solution $a$ of the ODE
	\begin{align*}
		-a'' - (a'^3)' =a
	\end{align*}
	with minimal prescribed period $T\in (0,2\pi)$. Then $w(x,t) \coloneqq a(t)b(x)$ satisfies \eqref{quasi}. Note that $h$ is sign-changing and $w$ is not exponentially localized. We found this solution by inserting the ansatz for $w$ with separated variables into \eqref{quasi}. We then defined $b(x)\coloneqq(1+x^2)^{-1/2}$ and set $g(x)\coloneqq -b''(x)/b(x)$ and $h(x)\coloneqq -b''(x)/b(x)^3$. The remaining equation for $a$ then turned out to be the above one.
\medskip

	The paper is structured as follows: In Section~\ref{variational_approach} we develop the variational setting and give the proof of Theorem~\ref{w is a weak solution general}. The proof of the additional regularity results of Theorem~\ref{w is even more regular} is given in Section~\ref{further_regularity}. In Section~\ref{infinitely_many_breathers} we give the proof of Theorem~\ref{multiplicity abstract} on the existence of infinitely many different breathers. In Section~\ref{approximation} we show that our breathers can be well approximated by truncation of the Fourier series in time. Finally, in the Appendix we give details on the background and proof of Theorem~\ref{step} (Section~\ref{details_example_step}) and Theorem~\ref{w is a weak solution in expl exa} (Section~\ref{explicit example Bloch Modes_WR}) as well as a technical detail on a particular embedding of H\"older spaces into Sobolev spaces (Section~\ref{embedding}).
	
\section{Variational Approach and Proof of Theorem~\ref{w is a weak solution general}} \label{variational_approach}
	The main result of our paper is Theorem~\ref{w is a weak solution general} which will be proved in this section. It is a consequence of Lemma~\ref{breathers} and Theorem~\ref{J attains a minimum and its properties} below. 
\medskip
	
	Formally \eqref{quasi} is the Euler-Lagrange-equation of the functional
	\begin{equation} \label{def_I}
		I(w)\coloneqq\int_D-\frac{1}{2}g(x)\abs{\partial_tw}^2+\frac{1}{2}\abs{\partial_xw}^2\dd{(x,t)} -\frac{1}{4}\gamma\int_{0}^{T}\abs{\partial_tw(0,t)}^4\dd{t}
	\end{equation}
	defined on a suitable space of $T$-periodic functions. Instead of directly searching for a critical point of this functional  we first rewrite the problem into a nonlinear Neumann boundary value problem under the assumption that $w$ is even in $x$. In this case \eqref{quasi} amounts to the following linear wave equation on the half-axis with nonlinear Neumann boundary conditions:
	\begin{gather}
		\begin{cases}
		g(x) w_{tt}-w_{xx}=0 & \text{for } (x,t)\in(0,\infty)\times\R,\\
		2w_x(0_+,t)=\gamma\left(w_t(0,t)^3\right)_t & \text{for }t\in\R
		\end{cases}\label{nonlinNeuBVP}
	\end{gather}
	where solutions $w\in\SobH{1}{[0,\infty)\times\T_T}$ with $\partial_tw(0,\cdot)\in\Leb{3}{\T_T}$ of \eqref{nonlinNeuBVP} are understood in the sense that 
	\begin{align}
		2\int_{D_+}-g(x)\partial_tw\,\partial_t\psi +\partial_xw\,\partial_x\psi\dd{(x,t)} -\gamma\int_{0}^{T} (\partial_t w(0,t))^3 \partial_t \psi(0,t)\dd{t}=0 \label{WeakEquation for nlinNeuBVP}
	\end{align}
	for all $\psi\in\Contc{\infty}{[0,\infty)\times\T_T}$ with $D_+=(0,\infty)\times\T_T$. It is clear that evenly extended solutions $w$ of 
	\eqref{WeakEquation for nlinNeuBVP} also satisfy \eqref{WeakEquation for (quasi)}. To see this note that every $\psi\in\Contc{\infty}{\R\times\T_T}$ can be split into an even and an odd part $\psi=\psi_{e}+\psi_{o}$ both belonging to $\Contc{\infty}{\R\times\T_T}$. Testing with $\psi_o$ in \eqref{WeakEquation for (quasi)} produces zeroes in all spatial integrals due to the evenness of $w$ and also in the temporal integral since $\psi_{o}(0,\cdot)\equiv 0$ due to oddness. Testing with $\psi_e$ in \eqref{WeakEquation for (quasi)} produces twice the spatial integrals appearing in \eqref{WeakEquation for nlinNeuBVP}. In the following we concentrate on finding solutions of \eqref{nonlinNeuBVP} for the linear wave equation with nonlinear Neumann boundary conditions.
	
	Motivated by the linear wave equation in \eqref{nonlinNeuBVP} we make the ansatz that 
	\begin{equation} \label{ansatz}
		w(x,t)=\sum_{k\in\Zodd}\frac{\hat{\alpha}_k}{k}\Phi_k(\abs{x})e_k(t),
	\end{equation}
	where $e_k(t)\coloneqq\frac{1}{\sqrt{T}}\e^{\i\omega kt}$ denotes the $\Leb{2}{\T_T}$-orthonormal Fourier base of $\T_T$, and where $\Phi_k$ are the decaying fundamental solutions $\Phi_k$ of $L_k$, cf. Lemma~\ref{exp_decaying_sol}. Such a function $w$ will always solve the linear wave equation in \eqref{nonlinNeuBVP} and we will determine real sequences $\hat{\alpha}= (\hat\alpha_k)_{k\in \Zodd}$ such that the nonlinear Neumann condition is satisfied as well. The additional factor $\frac{1}{k}$ is only for convenience, since $\partial_t$ generates a multiplicative factor $\i\omega k$.
\medskip
	
	The convolution between two sequences $\hat{z},\hat{y}\in\R^\Z$ is defined pointwise (whenever it converges) by $(\hat{z}*\hat{y})_k\coloneqq \sum_{l\in\Z}\hat{z}_l\hat{y}_{k-l}$. 
\medskip
		
	In order to obtain real-valued functions $w$ by the ansatz \eqref{ansatz} we require the sequence $\hat{\alpha}$ to be real and odd in $k$, i.e., $\hat{\alpha}_k\in \R$ and  $\hat{\alpha}_k = -\hat{\alpha}_{-k}$. Since \eqref{ansatz} already solves the wave equation in \eqref{nonlinNeuBVP}, it remains to find $\hat{\alpha}$ such that 
	\begin{align*}
		2w_x(0_+,t) = 2\sum_{k\in \Zodd} \frac{\hat{\alpha}_k}{k} \Phi_k'(0)e_k(t) \stackrel{!}{~=~} 	\frac{1}{T}\sum_{k\in \Zodd} \gamma\omega^4 k (\hat\alpha*\hat\alpha*\hat\alpha)_k e_k(t) = \gamma(w_t(0,t)^3)_t,
	\end{align*}
	where we have used $\Phi_k(0)=1$. As the above identity needs to hold for all $t\in \R$ we find
	\begin{equation} \label{euler_lagrange_alpha}
		(\hat\alpha*\hat\alpha*\hat\alpha)_k = \frac{2T\Phi_k'(0)}{\gamma\omega^4 k^2} \hat \alpha_k \quad\mbox{ for all } k \in \Zodd.
	\end{equation}
	This will be accomplished by searching for critical points $\hat{\alpha}$ of the functional
	\begin{align*}
		J(\hat{z})\coloneqq\frac{1}{4} (\hat{z}*\hat{z}*\hat{z}*\hat{z})_0+\frac{T}{\gamma \omega^4}\sum_k\frac{\Phi'_k(0)}{k^2}\hat{z}_k^2.
	\end{align*}
	defined on a suitable Banach space of real sequences $\hat{z}$ with $\hat{z}_k = -\hat{z}_{-k}$. Indeed, computing (formally) the Fr\'{e}chet derivative of $J$ at $\hat{\alpha}$ we find 
	\begin{equation}
		J'(\hat{\alpha})[\hat{y}]=\left(\hat{\alpha}*\hat{\alpha}*\hat{\alpha}*\hat{y}\right)_0+\frac{2T}{\gamma\omega^4}\sum_k\frac{\Phi'_k(0)}{k^2}\hat{\alpha}_k\hat{y}_k.  \label{frechet} 
	\end{equation}
	Let us indicate how \eqref{frechet} amounts to \eqref{euler_lagrange_alpha}. For fixed $k_0\in \Zodd$ we define the test sequence $\hat{y}\coloneqq(\delta_{k,k_0}-\delta_{k,-k_0})_{k\in \Zodd}$ which has exactly two non-vanishing entries at $k_0$ and at $-k_0$. Thus, $\hat{y}$ belongs to the same space of odd, real sequences as $\hat{\alpha}$ and can therefore be used as a test sequence in $J'(\hat{\alpha})[\hat{y}]=0$. After a short calculation using $\hat{\alpha}_k=-\hat{\alpha}_{-k}$, $\Phi_k'=\Phi_{-k}'$ we obtain  \eqref{euler_lagrange_alpha} for $k_0$.
\medskip
	
	It turns out that a real Banach space of real-valued sequences which is suitable for $J$ can be given by 
	\begin{align*}
		\Dom{J}\coloneqq\left\{ \hat{z}\in\R^\Zodd ~\big|~ \NORM{\hat{z}}<\infty,~ \hat{z}_k=-\hat{z}_{-k} \right\} \mbox{ where } \NORM{\hat{z}}\coloneqq \norm{\hat{z}*\hat{z}}_\seq{2}^\frac{1}{2}.
	\end{align*}	
	The relation between the function $I$ defined in \eqref{def_I} and the new functional $J$ is formally given by
	\begin{align*}
		I\left(\sum_{k\in\Zodd}\frac{\hat{z}_k}{k}\Phi_k(\abs{x})e_k(t)\right) =-\frac{\gamma\omega^4}{T} J\left(\hat{z}\right).
	\end{align*}
	
	\begin{lemma} \label{Charakterization Dom(J)}
		The space $(\Dom{J},\NORM{\cdot})$ is a separable, reflexive, real Banach space and isometrically embedded into the real Banach space $\Leb{4}{\T_T,\i\R}$ of purely imaginary-valued measurable functions. Moreover for $\hat{u}, \hat{v}, \hat{w}, \hat{z} \in \Dom{J}$ we have 
		\begin{align} 
		(\hat{u}*\hat{u}*\hat{u}*\hat{u})_0  & = \NORM{\hat{u}}^4, \\
		\abs{(\hat{u}*\hat{v}*\hat{w}*\hat{z})_0}  & \leq \NORM{\hat{u}}\,\NORM{\hat{v}}\,\NORM{\hat{w}}\,\NORM{\hat{z}}, \label{conv_multilinear} \\
		\norm{\hat{z}}_{\seq{2}} &\leq\NORM{\hat{z}}. \label{l2_l4}
		\end{align}
	\end{lemma}
	
	\begin{proof}
		We first recall the correspondence between real-valued sequences $\hat{z}\in l_2$ with $\hat{z}_k=-\hat{z}_{-k}$ and purely imaginary-valued functions $z\in\Leb{2}{\T_T,\i\R}$ by setting
		\begin{align*}
			\hat{z}_k\coloneqq\skp{z}{e_k}_{L^2(\T_T)} \mbox{ and } z(t)\coloneqq\sum_{k\in\Z}\hat{z}_ke_k(t)
		\end{align*}
		Parseval's identity provides the isomorphism $\norm{z}_{\Leb{2}{\T_T}}=\|\hat{z}\|_\seq{2}$. The following identity	
		\begin{align*}
			T\norm{z}_{\Leb{4}{\T_T}}^4 = T\int_0^T z(t)^4\,dt = (\hat{z}*\hat{z}*\hat{z}*\hat{z})_0 = \|\hat{z}*\hat{z}\|_\seq{2}^2 = \NORM{\hat{z}}^4
		\end{align*}
		shows that $\NORM{\cdot}$ is indeed a norm on $\Dom{J}$ and its provides the isometric embedding of $\Dom{J}$ into a subspace of $L^4(\T_T,\i\R)$. By Parseval's equality and H\"older's inequality we see that
		\begin{align*}
			\norm{\hat{z}}_{\seq{2}}=\norm{z}_{\Leb{2}{\T_T}}\leq 	T^\frac{1}{4}\norm{z}_{\Leb{4}{\T_T}}=\NORM{\hat{z}}
		\end{align*}
		so that $\Dom{J}$ is indeed a subspace of $l^2$. Finally, for any $\hat{u}, \hat{v}, \hat{w}, \hat{z} \in \Dom{J}$ we see that
		\begin{align*}
		    \abs{(\hat{u}*\hat{v}*\hat{w}*\hat{z})_0}  
		    = T \abs{\int_0^T u(t)v(t)w(t)z(t) \,dt}
		    \leq T\norm{u}_{L^4}\norm{v}_{L^4}\norm{w}_{L^4}\norm{z}_{L^4} 
		    = \NORM{\hat{u}}\,\NORM{\hat{v}}\,\NORM{\hat{w}}\,\NORM{\hat{z}}.
		\end{align*}
		This finishes the proof of the lemma.
	\end{proof}

	For $\frac{T}{2}$-anti-periodic functions $\psi\colon D\to \R$ of the space-time variable $(x,t)\in D$ we use the notation 
	\begin{equation} \label{ansatz_phi}
		\psi(x,t)=\sum_{k\in\Zodd}\hat{\psi}_k(x)e_k(t) = \sum_{k\in\Zodd} \frac{1}{k} \Psi_k(x) e_k(t)
	\end{equation}
	with $\frac{1}{k}\Psi_k(x)=\hat{\psi}_k(x)\coloneqq\skp{\psi(x,\cdot)}{e_k}_\Leb{2}{\T_T}$. The Parseval identity and the definition of $\NORM{\cdot}$ immediately lead to the following lemma.
		
	\begin{lemma} \label{characterization}
		For $\psi:D\to \R$ as in \eqref{ansatz_phi} the following holds:
		\begin{itemize}
			\item[(i)] $\|\psi_x\|_{L^2(D)}^2=\sum_{k} \frac{1}{k^2} \|\Psi_k'\|_{\Leb{2}{\R}}^2$,
			\item[(ii)] $\|\psi_t\|_{L^2(D)}^2=\omega^2\sum_{k} \|\Psi_k\|_{\Leb{2}{\R}}^2$,
			\item[(iii)] $T\|\psi_t(0,\cdot)\|_\Leb{4}{\T_T}^4 = \omega^4\NORM{\hat{y}}^4$ where $\hat{y}_k = \Psi_k(0)$ for $k\in \Zodd$.
		\end{itemize}
	\end{lemma}
	
	The next result give some estimates on the growth of norms of $\Phi_k$. It serves as a preparation for the proof of regularity properties for functions $w$ as in \eqref{ansatz} stated in Lemma~\ref{breathers}.
	
    \begin{lemma} \label{norm_estimates} 
    	Assume \eqref{C0}, \eqref{spectralcond} and \eqref{FurtherCond_phik}. Then 
	    \begin{equation} \label{phi_k_estimates}
	    	\|\Phi_k\|_{L^2(0,\infty)}= O(1), \quad \|\Phi_k'\|_{L^2(0,\infty)}=O(k), \quad \|\Phi_k'\|_{L^\infty(0,\infty)} = O(k^\frac{3}{2}).
	    \end{equation}
	    In particular $|\Phi_k'(0)|=O(k^\frac{3}{2})$.
	    \label{asymptotik}
	\end{lemma}
	\begin{proof} 
		The first part of \eqref{phi_k_estimates} is a direct consequence of \eqref{FurtherCond_phik}. 
	\medskip
	
	    Multiplying $L_k\Phi_k=0$ with $\Phi_k$, $\Phi_k'$ and integrating from $a\geq 0$ to $\infty$ we get 
	    \begin{align}
	    	\int_a^\infty -\omega^2 k^2 g(x)\Phi_k(x)^2+\Phi_k'(x)^2\,dx & = -\Phi_k(a)\Phi_k'(a), \label{mult1}\\
	    	\int_a^\infty -2 \omega^2 k^2 g(x) \Phi_k(x)\Phi_k'(x)\,dx & = -\Phi_k'(a)^2, \label{mult2}
	    \end{align}
	    respectively. Applying the Cauchy-Schwarz inequality to \eqref{mult2} and using the first part of \eqref{phi_k_estimates} we find 
	    \begin{equation} \label{mult3}
	    	\|\Phi_k'\|_{L^\infty(0,\infty)}^2 \leq O(k^2) \|\Phi_k'\|_{L^2(0,\infty)}
	    \end{equation}
	    and from \eqref{mult1}, \eqref{mult3} we get 
	    \begin{align*}
		    \|\Phi_k'\|_{L^2(0,\infty)}^2 & \leq O(k^2) + \|\Phi_k\|_{L^\infty(0,\infty)} 	\|\Phi_k'\|_{L^\infty(0,\infty)}\\
		    & \leq  O(k^2) + \|\Phi_k\|_{L^\infty(0,\infty)} O(k) 	\|\Phi_k'\|_{L^2(0,\infty)}^\frac{1}{2}.
	    \end{align*}
	    The $L^\infty$-assumption in \eqref{FurtherCond_phik} leads to 
	    \begin{align*}
	 	   \|\Phi_k'\|_{L^2(0,\infty)}^2  \leq O(k^2) + O(k)\|\Phi_k'\|_{L^2(0,\infty)}^\frac{1}{2} \leq O(k^2) +C_\epsilon O(k^\frac{4}{3}) + \epsilon \|\Phi_k'\|_{L^2(0,\infty)}^2,
		\end{align*}
	    where we have used Young's inequality with exponents $4/3$ and $4$. This implies the second inequality in \eqref{phi_k_estimates}.  Inserting this into \eqref{mult3} we obtain the third inequality in \eqref{phi_k_estimates}. 
    \end{proof} 

	\begin{lemma}\label{breathers}
		Assume \eqref{C0}, \eqref{spectralcond} and \eqref{FurtherCond_phik}. For $\hat{\alpha}\in\Dom{J}$ and $w:D\to\R$ as in \eqref{ansatz} we have $w_x, w_t \in L^2(D)$, $w_t(0,\cdot)\in\Leb{4}{\T_T}$ and there are values $C>0$ and $\rho>0$ such that $\abs{w(x,t)}\leq C\e^{-\rho\abs{x}}$.
	\end{lemma}
	\begin{rmk}
		The lemma does not require $\hat{\alpha}$ to be a critical point of $J$. The smoothness and decay of $w$ as in \eqref{ansatz} is simply a consequence of $\hat{\alpha} \in \Dom{J}$ and \eqref{FurtherCond_phik}. 
	\end{rmk}

	\begin{proof} 
		We use the characterization from Lemma~\ref{characterization}. Let us begin with the estimate for $\norm{\partial_t w}_{\Leb{2}{D}}$. By Lemma~\ref{norm_estimates} we have $\sup_k \|\Phi_k\|_{L^2(0,\infty)}<\infty$ so that 
		\begin{align*}
			\norm{\partial_t w}_{\Leb{2}{D}}^2 &=  2\omega^2 \sum_k \hat{\alpha}_k^2\norm{\Phi_k}_\Leb{2}{0,\infty}^2
			\leq 2\omega^2\Bigl(\sup_k\norm{\Phi_k}_\Leb{2}{0,\infty}\Bigr)^2 \norm{\hat{\alpha}}_\seq{2}^2 \\
			& \leq 2\omega^2\Bigl(\sup_k\norm{\Phi_k}_\Leb{2}{0,\infty}\Bigr)^2 \NORM{\hat{\alpha}}^2<\infty,
		\end{align*}
		which finishes our first goal. Next we estimate $\norm{\partial_x w}_{\Leb{2}{D}}$. Here we use again Lemma~\ref{norm_estimates} to find
		\begin{align*}
			\norm{\partial_x w}_{\Leb{2}{D}}^2 =  2\sum_k\frac{\hat{\alpha}_k^2}{k^2}\norm{\Phi_k'}_\Leb{2}{0,\infty}^2 \leq C\|\hat\alpha\|_{l^2}^2 \leq C\NORM{\hat\alpha}^2<\infty
        \end{align*}
		which finishes our second goal. Next we show that $w_t(0,\cdot)\in\Leb{4}{\T_T}$. Using $\Phi_k(0)=1$ we observe that
		\begin{align*}
			T\norm{w_t(0,\cdot)}_\Leb{4}{\T_T}^4 	=T\int_{0}^T\Bigl(\sum_{k\in\Zodd}\i\omega\hat{\alpha}_k\Phi_k(0)e_k(t)\Bigr)^4\dd{t}\\
			=\omega^4\NORM{\hat{\alpha}}^4<\infty.
		\end{align*}
		Finally we show the uniform-in-time exponential decay of $w$. By construction $w$ is even in $x$, hence we only consider $x>0$. By \eqref{FurtherCond_phik} we see that 
		\begin{align*}
			\abs{w(x,t)} \leq\sum_k\frac{\abs{\hat{\alpha}_k}}{\abs{k}}\abs{\Phi_k(x)} 
			=\sum_k\frac{\abs{\hat{\alpha}_k}}{\abs{k}}Ce^{-\alpha x} 
			\leq \|\hat\alpha\|_{l^2}\left(\sum_k \frac{1}{k^2}\right)^{1/2} C e^{-\alpha x} \leq \tilde C e^{-\alpha x}
		\end{align*}
    	which finishes the proof of the lemma.
	\end{proof}
	
	In the following result we will show that minimizers of $J$ on $\Dom{J}$ exist, are solutions of \eqref{euler_lagrange_alpha} and indeed correspond to weak solutions of \eqref{quasi}.
	
	\begin{thm}\label{J attains a minimum and its properties}
		Assume \eqref{C0}, \eqref{spectralcond} and \eqref{FurtherCond_phik}. Then the functional $J$ is well defined on its domain $\Dom{J}$, Fr\'{e}chet-differentiable, bounded from below and attains its negative minimum provided 
		\begin{itemize}
			\item[(i)] $\gamma<0$ and the sequence $\left(\Phi'_k(0)\right)_{k\in\Nodd}$ has at least one positive element, or
			\item[(ii)] $\gamma>0$ and the sequence $\left(\Phi'_k(0)\right)_{k\in\Nodd}$ has at least one negative element. 
		\end{itemize}
		For every critical point $\hat{\alpha}\in\Dom{J}$ the corresponding function $w(x,t)\coloneqq\sum_{k\in\Zodd}\frac{\hat{\alpha}_k}{k}\Phi_k(\abs{x})e_k(t)$ is a nontrivial weak solution of \eqref{quasi}.
	\end{thm}
	
	\begin{proof} 
		Note that $J(\hat z) = \frac{1}{4} \NORM{\hat{z}}^4 + J_1(\hat z)$, where $J_1(\hat z)= \sum_k a_k \hat z_k^2$ with $a_k = \frac{T\Phi_k'(0)}{\gamma\omega^4 k^2}$. By Lemma~\ref{norm_estimates} the sequence $(a_k)_k$ is converging to $0$ as $|k|\to \infty$, so in particular it is bounded. Due to \eqref{l2_l4} one finds that $J$ is well defined and continuous on $\Dom{J}$, and moreover, that for $\hat{z}\in \Dom{J}$
		\begin{align*}
			J(\hat{z}) \geq\frac{1}{4}\NORM{\hat{z}}^4 -\sup_k |a_k| \sum_k\hat{z}_k^2
			\geq\frac{1}{4}\NORM{\hat{z}}^4 - \sup_k |a_k|\NORM{\hat{z}}^2.
		\end{align*}
		This implies that $J$ is coercive and bounded from below. The weak lower semi-continuity of $J$ follows from the convexity and continuity of the map $\hat z \mapsto \NORM{\hat{z}}^4$ and the weak continuity of $J_1$. To see the latter take an arbitrary $\epsilon>0$. Then there is $k_0\in \N$ such that $|a_k|\leq \epsilon$ for $|k|> k_0$ and this implies the inequality 
		\begin{align} \label{estimate_J_1}
			|J_1(\hat z)- J_1(\hat{y})| \leq  \sup_k |a_k| \sum_{|k|\leq k_0} |\hat z_k^2-\hat{y_k}^2| + \epsilon (\|\hat z\|_{l^2}^2 + \|\hat{y}\|_{l^2}^2) \quad\forall\,\hat{z},\hat{y}\in\Dom{J}. 
		\end{align}
		Since $(\Dom{J},\NORM{\cdot})$ continuously embeds into $l^2$ any weakly convergent sequence in $(\Dom{J},\NORM{\cdot})$ also weakly converges in $l^2$ and in particular pointwise. This pointwise convergence together with the boundedness of the sequence and \eqref{estimate_J_1} yields the weak continuity of $J_1$ and thus the weak lower semi-continuity of $J$. As a consequence, cf. Theorem 1.2 in \cite{struwe}, we get the existence of a minimizer.  
	
		In order to check that the minimizer is nontrivial is suffices to verify that $J$ attains negative values. Here we distinguish between case (i) and (ii) in the assumptions of the theorem. In case (i) when $\gamma<0$ we find an index $k_0$ such that $\Phi_{k_0}'(0)>0$. In case (ii) when $\gamma>0$ we choose $k_0$ such that $\Phi_{k_0}'(0)<0$. In both cases we obtain that $\Phi_{k_0}'(0)/\gamma<0$. If we set $\hat{y}\coloneqq(\delta_{k,k_0}-\delta_{k,-k_0})_{k\in \Zodd}$ then $\hat{y}$ has exactly two non-vanishing entries, namely $+1$ at $k_0$ and $-1$ at $-k_0$. Hence $\hat{y}\in\Dom{J}$. Using the property $\Phi_{k_0}'=\Phi_{-k_0}'$ we find for $t\in \R$ 
		\begin{align*}
			J(t \hat{y})=t^4\frac{1}{4}\NORM{\hat{y}}^4 +2t^2\frac{T\Phi'_{k_0}(0)}{\gamma\omega^4k_0^2} 
		\end{align*}
		which is negative by the choice of $k_0$ provided $t>0$ is sufficiently small. Thus, $\inf_{\Dom{J}}J<0$ and every minimizer $\hat\alpha$ is nontrivial.

		Next we show for every critical point $\hat\alpha$ of $J$ that $w(x,t)\coloneqq\sum_{k\in\Zodd}\frac{\hat{\alpha}_k}{k}\Phi_k(\abs{x})e_k(t)$ is a weak solution of \eqref{quasi}. The regularity properties $w\in\SobH{1}{\RT}$, $\partial_tw(0,\cdot)\in\Leb{4}{\T_T}$ and the exponential decay have already been shown in Lemma~\ref{breathers}. We skip the standard proof that $J\in\Cont{1}{\Dom{J},\R}$ and that its Fr\'{e}chet-derivative is given by \eqref{frechet}. We will show that \eqref{WeakEquation for (quasi)} holds for any $\psi$ as in \eqref{ansatz_phi} with even functions $\Psi_k\in H^1(\R)$, $\Psi_k=-\Psi_{-k}$ such that $\psi_x, \psi_t \in L^2(D)$ and $\psi(0,\cdot)\in L^4(\T_T)$ as described in Lemma~\ref{characterization}. We begin by deriving expressions and estimates for the functionals 
		\begin{align*}
			H_1(\psi) = \int_D g(x) w_t \psi_t\,d(x,t), \quad 
			H_2(\psi) = \int_D w_x \psi_x \,d(x,t), \quad 
			H_3(\psi) = \int_{0}^{T} w_t(0,t)^3\psi_t(0,t)\,dt.
		\end{align*}
		In a first step we assume that the sum in \eqref{ansatz_phi} is finite in order to justify the exchange of summation and integration in the following. Then, starting with $H_1$ we find 
		\begin{align*}
			H_1(\psi) 
			&=  -\omega^2 \int_D g(x)\sum_{k,l} \hat{\alpha}_k\Phi_k(\abs{x}) \Psi_l(|x|) 	e_k(t)e_l(t) \dd{(x,t)} \\
			&=-2\omega^2 \sum_k \hat{\alpha}_k\int_0^\infty g(x)\Phi_k(x)\Psi_{-k}(x)\dd{x}\\
			&=2\omega^2 \sum_k \hat{\alpha}_k\int_0^\infty g(x)\Phi_k(x)\Psi_k(x)\dd{x}, 
			\\
			|H_1(\psi)| &\leq 2\omega^2 \|g\|_{L^\infty(\R)} \Bigl(\sum_k \hat{\alpha}_k^2 \|\Phi_k\|^2_{L^2(0,\infty)}\Bigr)^\frac{1}{2} \Bigl(\sum_k\|\Psi_k\|_{L^2(0,\infty)}^2\Bigr)^\frac{1}{2}= \|g\|_{L^\infty(\R)} \|w_t\|_{L^2(D)}\|\psi_t\|_{L^2(D)}
		\end{align*}
		and similarly for $H_2$ we find using \eqref{eq:bloch}
		\begin{align*}
			H_2(\psi) &=  \int_D \sum_{k,l} \frac{\hat{\alpha}_k}{k}\Phi_k'(\abs{x}) 	\frac{1}{l}\Psi_l'(|x|) e_k(t)e_l(t) \dd{(x,t)} \\
			&= 2\sum_k \frac{\hat{\alpha}_k}{-k^2} \int_0^\infty \Phi_k'(x)\Psi_{-k}'(x)\dd{x}\\
			&= 2\sum_k \frac{\hat{\alpha}_k}{k^2} \int_0^\infty \Phi_k'(x)\Psi_k'(x)\dd{x} \\
			&= 2\omega^2\sum_k \hat{\alpha}_k \int_0^\infty g(x) \Phi_k(x)\Psi_k(x)\dd{x} - 2\sum_k \frac{\hat{\alpha}_k}{k^2}\Phi_k'(0)\Psi_k(0), \\  
			|H_2(\psi)| &\leq 2\Bigl(\sum_k \frac{\hat{\alpha}_k^2}{k^2}\|\Phi_k'\|_{L^2(0,\infty)}^2 \Bigr)^\frac{1}{2} \Bigl(\sum_k \frac{1}{k^2}\|\Psi_k'\|_{L^2(0,\infty)}^2 \Bigr)^\frac{1}{2} = \|w_x\|_{L^2(D)}\|\psi_x\|_{L^2(D)}. 
		\end{align*}
		Moreover, considering $H_3$ and setting $\hat{y}_k \coloneqq \Psi_k(0)$ for $k\in \Zodd$ one sees 
		\begin{align*}
			H_3(\psi) &= \omega^4 \int_{0}^{T}\Bigl(\sum_k \hat{\alpha}_ke_k(t)\Bigr)^3\Bigl(\sum_l \Psi_l(0)e_l(t)\Bigr)\dd{t} \\
			&= \frac{\omega^4}{T} (\hat{\alpha}*\hat{\alpha}*\hat{\alpha}*\hat{y})_0,\\
			\abs{H_3(\psi)} & \leq \frac{\omega^4}{T} \NORM{\hat{\alpha}}^3  \NORM{\hat{y}}
			= \norm{w_t(0,\cdot)}_\Leb{4}{\T_T}^3\norm{\psi_t(0,\cdot)}_\Leb{4}{\T_T}.
		\end{align*}
		Hence $H_1, H_2$ and $H_3$ are bounded linear functionals of the variable $\psi$ as in \eqref{ansatz_phi} with $\psi_x, \psi_t \in L^2(D)$ and $\psi_t(0,\cdot)\in \Leb{4}{\T_T}$. For such $\psi$ we use the above formulae for $H_1, H_2, H_3$ and compute the linear combination
		\begin{align*}
			-H_1(\psi)+H_2(\psi)-\gamma H_3(\psi) = -2\sum_{k} \frac{\hat{\alpha}_k}{k^2} 	\Phi_k'(0)\Psi_k(0) - \frac{\gamma \omega^4}{T} (\hat{\alpha}*\hat{\alpha}*\hat{\alpha}*\hat{y})_0=0
		\end{align*}
		due to the Euler-Lagrange equation for the functional $J$, i.e., the vanishing of $J'(\hat{\alpha})[\hat{y}]$ in \eqref{frechet} for all $\hat{y} \in \Dom{J}$. The last equality means that $w$ is a weak solution of \eqref{quasi}.	
	\end{proof}

\section{Further Regularity} \label{further_regularity}
	Here we prove Theorem~\ref{w is even more regular}. We observe first that in the example of a periodic step-potential in Theorem~\ref{w is a weak solution in expl exa} we find that not only $\Phi'_k(0)=O(k^\frac{3}{2})$ holds (as Lemma~\ref{norm_estimates} shows) but even $\Phi'_k(0)=O(k)$ is satisfied. It is exactly this weaker growth that we can exploit in order to prove additional smoothness of the solutions of \eqref{quasi}. We begin by defining for $\nu>0$ the Banach space of sequences 
	\begin{align*}
		\seqsobh{\nu}\coloneqq\Bigl\{\hat{z}\in\seq{2} \mbox{ s.t. } \|\hat{z}\|_{h^\nu}^2 \coloneqq \sum_k (1+k^2)^{\nu}\abs{\hat{z}_k}^2<\infty\Bigr\}.
	\end{align*}
	Moreover, we use the isometric isomorphism between $h^\nu$ and 
	\begin{align*}
		H^{\nu}(\T_T) = \Bigl\{z(t)=\sum_k \hat{z}_k e_k(t) \text{ s.t. } \hat{z}\in\seqsobh{\nu} \Bigr\}
	\end{align*}
	by setting $\|z\|_{H^\nu} \coloneqq \|\hat z\|_{h^\nu}$. We also use the Morrey embedding $\SobH{1+\nu}{\T_T}\to \Cont{0,\frac{1}{2}+\nu}{\T_T}$ for $\nu \in (0,1/2)$ and the following embedding: $\Cont{0,\nu}{\T_T}\to \SobH{\tilde\nu}{\T_T}$ for $0<\tilde\nu<\nu\leq 1$, cf. Lemma~\ref{unusual_embedd} in the Appendix. The latter embedding means that $\hat{z} \in h^{\tilde\nu}$ provided $z\in\Cont{0,\nu}{\T_T}$ and $0<\tilde\nu<\nu\leq 1$.
	
	
	\begin{thm}\label{smoothness alpha}
		Assume \eqref{C0}, \eqref{spectralcond}, \eqref{FurtherCond_phik} and in addition $\Phi'_k(0) = O(k)$. For every $\hat{\alpha}\in \Dom{J}$ with $J'(\hat{\alpha})=0$ we have $\hat{\alpha}\in h^\nu$ for every $\nu\in (0,1/4)$. 
	\end{thm} 
	\begin{proof}
		Let $\hat{\alpha}\in \Dom{J}$ with $J'(\hat{\alpha})=0$. Recall from \eqref{euler_lagrange_alpha} that
		\begin{equation} \label{el}
			(\hat{\alpha}*\hat{\alpha}*\hat{\alpha})_k = \hat{\eta}_k \hat{\alpha}_k \quad\mbox{ where }\quad \hat{\eta}_k \coloneqq \frac{2T\Phi'_k(0)}{\gamma\omega^4k^2} \mbox{ for } k \in \Zodd
		\end{equation}
		so that $|\hat{\eta}_k| \leq C/k$.
		If we define the convolution of two $T$-periodic functions $f,g\in\Leb{2}{\T_T}$ on the torus $\T_T$ as 
		\begin{align*}
			\left(f*g\right)(t)\coloneqq\frac{1}{\sqrt{T}}\int_{0}^{T}f(s)g(t-s)\dd{s}
		\end{align*}
		and if we set 
		\begin{align*}
			\alpha(t) \coloneqq \sum_k \hat{\alpha}_k e_k(t), \quad \eta (t) \coloneqq \sum_k \hat{\eta}_k e_k(t)
		\end{align*}
		then the equation 
		\begin{equation} \label{el_equiv}
			\alpha^3=\alpha*\eta
		\end{equation}
		for the $T$-periodic function $\alpha\in\Leb{4}{\T_T}$ is equivalent to the equation \eqref{el} for the sequence $\hat\alpha\in \Dom{J}$. We will analyze \eqref{el_equiv} with a bootstrap argument. 

	\emph{Step 1:} 
		We show that $\alpha \in\Cont{0,\frac{1}{6}}{\T_T}$. The right hand side of \eqref{el_equiv} is an $\SobH{1}{\T_T}$-function since 
		\begin{align*}
			\norm{\alpha*\eta}_\SobH{1}{\T_T}^2
			=\norm{\hat{\alpha}\hat{\eta}}_\seqsobh{1}^2
			\leq \sum_{k\in\Zodd} (1+k^2)\hat{\alpha}_k^2\frac{C^2}{k^2}
			\leq 2C^2 \|\hat{\alpha}\|_{l^2}^2 <\infty.
		\end{align*}
		Therefore, using \eqref{el_equiv} we see that $\alpha^3\in H^1(\T_T)$ and by the Morrey embedding that $\alpha^3\in\Cont{0,\frac{1}{2}}{\T_T}$. Since the inverse of the mapping $x\mapsto x^3$ is given by $x\mapsto |x|^{-\frac{2}{3}}x$, which is a $\cont{0,\frac{1}{3}}(\R)$-function, we obtain $\alpha\in\Cont{0,\frac{1}{6}}{\T_T}$.
				
	\emph{Step 2:} 
		We fix $q\in (0,1)$ and show that if $\alpha\in \Cont{0,\nu_n}{\T_T}$ for some $\nu_n\in (0,1/2)$ solves \eqref{el_equiv} then $\alpha \in \Cont{0,\nu_{n+1}}{\T_T}$ with $\nu_{n+1}= \frac{q\nu_n}{3}+\frac{1}{6}$. For the proof we iterate the process from Step 1 and we start with $\alpha\in\Cont{0,\nu_n}{\T_T}$. Then, according to Lemma~\ref{unusual_embedd} of the Appendix, $\alpha\in\SobH{q\nu_n}{\T_T}$ and hence $\hat{\alpha} \in \seqsobh{q\nu_n}$. Then as before the convolution of $\alpha$ with $\eta$ generates one more weak derivative, namely
		\begin{align*}
			\norm{\alpha*\eta}_\SobH{1+q\nu_n}{\T_T}^2
			=\norm{\hat{\alpha}\hat{\eta}}_\seqsobh{1+q\nu_n}^2
			\leq \sum_k(1+k^2)^{1+q\nu_n}\hat{\alpha}_k^2\frac{C^2}{k^2}
			\leq C^2 \|\hat{\alpha}\|_{h^{q\nu_n}}<\infty.
		\end{align*}
		Hence by \eqref{el_equiv} we conclude $\alpha^3\in\SobH{1+q\nu_n}{\T_T}$ and by the Morrey embedding $\alpha^3\in\Cont{0,\frac{1}{2}+q\nu_n}{\T_T}$ provided $q\nu_n \in (0,1/2)$. As in Step 1 this implies $\alpha\in\Cont{0,\nu_{n+1}}{\T_T}$ with $\nu_{n+1} =\frac{1}{6}+\frac{q\nu_n}{3}$. 
	\medskip
		
		Starting with $\nu_1=1/6$ from Step 1 we see by Step 2 that $\nu_n\nearrow\frac{1}{2(3-q)}$. Since $q\in (0,1)$ can be chosen arbitrarily close to $1$ this finishes the proof.
	\end{proof}
	
	With this preparation the proof of Theorem \ref{w is even more regular} is now immediate.
	
	\begin{proof}[Proof of Theorem~ \ref{w is even more regular}] Let $w(x,t) = \sum_{k\in \Zodd} \frac{\hat{\alpha}_k}{k} \Phi_k(|x|)e_k(t)$ with $\hat{\alpha} \in \Dom{J}$ such that $J'(\hat{\alpha})=0$. Recall from assumption \eqref{FurtherCond_phik} that $C\coloneqq \sup_k\norm{\Phi_k}_\Leb{2}{0,\infty}^2<\infty$. Likewise, from Lemma~\ref{norm_estimates} we have $\norm{\Phi_k'}_\Leb{2}{0,\infty}^2 \leq \tilde Ck^2$ for all $k\in \Zodd$ and some $\tilde C>0$. Therefore, using Theorem~\ref{smoothness alpha} we find for all $\nu<\frac{1}{4}$
	\begin{align*}
		\norm{\partial_t^{1+\nu} w}_\Leb{2}{D}^2
		=2\omega^{2+2\nu}\sum_k\hat{\alpha}_k^2|k|^{2\nu}\norm{\Phi_k}_\Leb{2}{0,\infty}^2
		\leq2\omega^{2+2\nu}C \|\hat{\alpha}\|_{h^\nu}^2 <\infty
	\end{align*}
	and likewise
	\begin{align*}
		\norm{\partial_t^\nu w_x}_\Leb{2}{D}^2
		=2\omega^{2\nu}\sum_k\hat{\alpha}_k^2|k|^{2\nu-2}\norm{\Phi_k'}_\Leb{2}{0,\infty}^2
		\leq 2\omega^{2\nu}\tilde C\|\hat{\alpha}\|_{h^\nu}^2 <\infty.
	\end{align*}
	This establishes the claim.
	\end{proof}
	
\section{Existence of Infinitely Many Breathers} \label{infinitely_many_breathers}
	In this section we extend Theorem~\ref{w is a weak solution general} by the following multiplicity result.
	
	\begin{thm}\label{multiplicity abstract}
		Assume \eqref{C0}, \eqref{spectralcond} and \eqref{FurtherCond_phik}. Then \eqref{quasi} has infinitely many nontrivial, $T$-periodic weak solution $w$ in the sense of Definition~\ref{Defn of weak Sol to (quasi)} with $T=\frac{2\pi}{\omega}$ provided 
		\begin{itemize}
			\item[(i)] $\gamma<0$ and there exists an integer $l_-\in \Nodd$ such that for infinitely many $j\in \N$ the sequence $\Bigl(\Phi'_{m\cdot l_-^j}(0)\Bigr)_{m\in\Nodd}$ has at least one positive element, 
			\item[(ii)] $\gamma>0$ and there exists an integer $l_+\in \Nodd$ such that for infinitely many $j\in \N$ the sequence $\Bigl(\Phi'_{m\cdot l_+^j}(0)\Bigr)_{m\in\Nodd}$ has at least one negative element. 
		\end{itemize}
	\end{thm}

	\begin{rmk} \label{remark_infinitely} 
		In the above Theorem, conditions \eqref{spectralcond} and \eqref{FurtherCond_phik} can be weakened: instead of requiring them for all $k\in \Nodd$ it suffices to require them for $k\in l_-^j\Nodd$, $k\in l_+^j\Nodd$ respectively. We prove this observation together with the one in Remark~\ref{remark_Dr} at the end of this section.
	\end{rmk}
	
	We start with an investigation about the types of symmetries which are compatible with our equation. The Euler-Lagrange equation \eqref{euler_lagrange_alpha} for critical points $\hat{\alpha}\in\Dom{J}$ of $J$ takes the form $(\hat{\alpha}*\hat{\alpha}*\hat{\alpha})_k = \hat{\eta}_k \hat{\alpha}_k$ with $\hat{\eta}_k \coloneqq \frac{2T\Phi'_k(0)}{\gamma\omega^4k^2}$ for $k \in \Zodd$. Next we describe subspaces of $\Dom{J}$ which are invariant under triple convolution and pointwise multiplication with $(\hat\eta_k)_{k\in \Zodd}$. It turns out that these subspaces are made of sequences $\hat{z}$ where only the $r^{th}$ entry modulus $2r$ is occupied.
	
	\begin{defn} 
		For $r\in \Nodd, p \in \Neven$ with $r<p$ let 
		\begin{align*}
			\Dom{J}_{r,p} = \{\hat{z}\in \Dom{J}:\forall\,k\in\Z, k\neq r ~\mathrm{mod}~p 	\colon\hat{z}_k=0 \}.
		\end{align*} 
	\end{defn}
	
	\begin{lemma} 
		For $r\in \Nodd, p\in \Neven$ with $r<p$ and $p\not = 2r$ we have $\Dom{J}_{r,p}=\{0\}$. 
	\end{lemma}
	\begin{proof} 
		Let $\hat{z}\in \Dom{J}_{r,p}$. For all $k\not \in r+p\Z$ we have $\hat{z}_k=0$ by definition of $\Dom{J}_{r,p}$. Let therefore $k=r+pl_1$ for some $l_1\in \Z$. Then $-k=-r-pl_1 \not \in r+p\Z$ because otherwise $2r=-p(l_1+l_2)=p|l_1+l_2|$ for some $l_2\in \Z$. Since by assumption $p>r$ we get $|l_1+l_2|<2$. But clearly $|l_1+l_2|\not \in \{0,1\}$ since $r\not= 0$ and $p\not = 2r$ by assumption. By this contradiction we have shown $-k\not \in r+p\Z$ so that necessarily $0=\hat z_{-k}=-\hat z_{k}$. This shows $\hat z=0$.  
	\end{proof}
	
	In the following we continue by only considering $\mathcal{D}_r \coloneqq\Dom{J}_{r,2r}$ for $r\in\Nodd$. 
	
	\begin{prop} \label{zwei_eingenschaften} 
		Let $r\in\Nodd$. 
		\begin{itemize}
			\item[(i)] 
				The elements $\hat z \in \mathcal{D}_r$ are exactly those elements of $\Dom{J}$ which generate $\frac{T}{2r}$-antiperiodic functions $\sum_{k\in \Zodd} \frac{\hat z_k}{k}\Phi_k(x)e_k(t)$. 
			\item[(ii)] 
				If $\hat z\in \mathcal{D}_r$ then $(\hat z*\hat z*\hat z)_k=0$ for all $k\not\in r+2r\Z$.  
		\end{itemize}
	\end{prop}
	\begin{proof} 
		(i) An element $\hat z\in \Dom{J}$ generates a $\frac{T}{2r}$-antiperiodic function $z(x,t)= \sum_{k\in \Zodd} \frac{\hat z_k}{k}\Phi_k(x)e_k(t)$ if and only if $z(x,t+\frac{T}{2r})=-z(x,t)$. Comparing the Fourier coefficients we see that this is the case if for all $k\in\Zodd$ we have $\hat z_k\bigl(\exp(\frac{\i\omega kT}{2r})+1\bigr)=0$, i.e., either $k\in r+2r\Z$ or $\hat z_k=0$. This is exactly the condition that $\hat z \in \mathcal{D}_r$. 
		\\
		(ii) Let $\hat z\in \mathcal{D}_r$ and assume that there is $k\in\Z$ such that $0\not = (\hat z*\hat z*\hat z)_k=\sum_{l,m} \hat z_l\hat z_{m-l}\hat z_{k-m}$. So there is $l_0, m_0\in \Zodd$ such that $\hat z_{l_0}, \hat z_{m_0-l_0}, \hat z_{k-m_0}\not =0$ which means by the definition of $\mathcal{D}_r$ that $l_0, m_0-l_0, k-m_0\in r+2r\Z$. Thus $k = l_0+m_0-l_0+k-m_0 \in 3r+2r\Z = r+2r\Z$.  
	\end{proof}
	
	\begin{proof}[Proof of Theorem~\ref{multiplicity abstract}] 
		We give the proof in case (i); for case (ii) the proof only needs a trivial modi\-fication. Let $r=l^j$ where $j$ is an index such that the sequence $\Bigl(\Phi'_{k\cdot l^j}(0)\Bigr)_{k\in \Nodd}$ has a positive element (we have changed the notation from $l_-$ to $l$ for the sake of readability). Since $\mathcal{D}_r$ is a closed subspace of $\Dom{J}$ we have as before in Theorem~\ref{J attains a minimum and its properties} the existence of a minimizer $\hat\alpha^{(r)}\in \mathcal{D}_r$, i.e., $J(\hat\alpha^{(r)})=\min_{\mathcal{D}_r}J<0$. Moreover, $\hat\alpha^{(r)}$ satisfies the restricted Euler-Lagrange-equation
		\begin{equation}
			0=J'\left(\hat\alpha^{(r)}\right)\left[\hat{x}\right]=\left(\hat\alpha^{(r)}*\hat\alpha^{(r)}*\hat\alpha^{(r)}*\hat{x}\right)_0+\frac{2T}{\gamma\omega^4}\sum_k\frac{\Phi'_k(0)}{k^2}\hat\alpha^{(r)}_k\hat{x}_k \qquad\forall\,\hat{x}\in \mathcal{D}_r.  \label{frechet_symmetric} 
		\end{equation}
		We need to show that \eqref{frechet_symmetric} holds for every $\hat z\in \Dom{J}$. If for an arbitrary $\hat{z}\in\Dom{J}$ we define $\hat{x}_k\coloneqq\hat{z}_k$ for $k\in r+2r\Z$ and $\hat{x}_k\coloneqq0$ else then $\hat{x}\in \mathcal{D}_r$. If we furthermore define $\hat{y}\coloneqq\hat{z}-\hat{x}$ then $\hat{y}_k=0$ for all $k\in r+2r\Z$. This implies in particular that 
		\begin{align*}
			\sum_k\frac{\Phi'_k(0)}{k^2}\hat\alpha^{(r)}_k\hat{y}_k = 0
		\end{align*} 
		and by using (ii) of Proposition~\ref{zwei_eingenschaften} also 
		\begin{align*} 
			(\hat\alpha*\hat\alpha*\hat\alpha*\hat y)_0=\sum_{k}\left(\hat\alpha^{(r)}*\hat\alpha^{(r)}*\hat\alpha^{(r)}\right)_k\hat{y}_{-k}=0.
		\end{align*}
		This implies $J'(\hat\alpha^{(r)})[\hat{y}]=0$ and since by \eqref{frechet_symmetric} also $J'(\hat\alpha^{(r)})[\hat{x}]=0$ we have succeeded in proving that $J'(\hat\alpha^{(r)})=0$. 
	\medskip
		
		It remains to show the multiplicity result. For this purpose we only consider $r=l^{j_m}$ for $j_m\to \infty$ as $m\to\infty$ where $j_m$ is an index such that the sequence $\Bigl(\Phi'_{l^{j_m}k}(0)\Bigr)_{k\in \Nodd}$ has a positive element. First we observe that $\mathcal{D}_{l^{j_m}}\supsetneq \mathcal{D}_{l^{j_{m+1}}}$. Assume for contradiction that the set $\{\hat\alpha^{(l^{j_m})}\}$ is finite. Then we have a subsequence $(j_{m_n})_{n\in \N}$ such that $\hat\alpha = \hat\alpha^{(l^{j_{m_n}})}$ is constant. But then 
		\begin{align*}
			\hat\alpha \in \bigcap_{n\in\N} \mathcal{D}_{l^{j_{m_n}}} = \bigcap_{j\in\N} \mathcal{D}_{l^j}=\{0\}.
		\end{align*}
		This contradiction shows the existence of infinitely many distinct critical points of the function $J$ and finishes the proof of the theorem.
	\end{proof}

	\begin{proof}[Proof of Remark~\ref{remark_Dr} and Remark~\ref{remark_infinitely}] 
		The proof of Theorem~\ref{multiplicity abstract} works on the basis that it suffices to minimize the functional $J$ on $\mathcal{D}_r$. In this way a $\frac{T}{2r}$-antiperiodic breather is obtained. For $\hat z\in \mathcal{D}_r$ only the entries $\hat z_k$ with $k\in r\Zodd$ are nontrivial while all other entries vanish. Therefore, \eqref{spectralcond} and \eqref{FurtherCond_phik} and the values of $\Phi_k'(0)$ are only relevant for $k\in r\Zodd$. In the special case of Remark~\ref{remark_infinitely} we take $r=l_\pm^j$.
	\end{proof}

\section{Approximation by Finitely Many Harmonics}\label{approximation}
	Here we give some analytical results on finite dimensional approximation of the breathers obtained in Theorem~\ref{w is a weak solution general}. The finite dimensional approximation is obtained by cutting-off the ansatz \eqref{ansatz} and only considering harmonics of order $|k|\leq N$. Here a summand in the series \eqref{ansatz} of the form $\Phi_k(|x|)e_k(t)$ is a called a harmonic since it satisfies the linear wave equation in \eqref{nonlinNeuBVP}. We will prove that $J$ restricted to spaces $\Dom{J^{(N)}}$ of cut-off ansatz functions still attains its minimum and that the sequence of the corresponding minimizers converges up to a subsequence to a minimizer of $J$ on $\Dom{J}$.
	\begin{defn}
		Let $N\in\Nodd$. Define
		\begin{align*}
			J^{(N)}\coloneqq J|_\Dom{J^{(N)}},\qquad \Dom{J^{(N)}}\coloneqq\left\lbrace \hat{z}\in\Dom{J} ~\big|~ \forall\,\abs{k}>N\colon\hat{z}_k=0 \right\rbrace
		\end{align*}
	\end{defn}
	\begin{lemma} \label{lemma_approximation}
        Under the assumptions of Theorem~\ref{w is a weak solution general} the following holds: 
		\begin{enumerate}
			\item[(i)] 
				For every $N\in \Nodd$ sufficiently large there exists $\hat{\alpha}^{(N)}\in\Dom{J^{(N)}}$ such that $J(\hat{\alpha}^{(N)})=\inf J^{(N)}<0$ and $\lim_{N\to\infty}J(\hat{\alpha}^{(N)})=\inf J$.
			\item[(ii)]  
				There is $\hat{\alpha}\in\Dom{J}$ such that up to a subsequence (again denoted by $(\hat{\alpha}^{(N)})_N$) we have
				\begin{align*}
					\hat{\alpha}^{(N)}\to\hat{\alpha} \qquad \text{ in }~\Dom{J}
				\end{align*} 
				and $J(\hat{\alpha})=\inf J$.
		\end{enumerate}
	\end{lemma}
	\begin{rmk}
		The Euler-Lagrange-equation for $\hat{\alpha}^{(N)}$ reads:
		\begin{align*}
			0=J'\left(\hat{\alpha}^{(N)}\right)[\hat{y}]=\left(\hat{\alpha}^{(N)}*\hat{\alpha}^{(N)}*\hat{\alpha}^{(N)}*\hat{y}\right)_0+\frac{2T}{\gamma\omega^4}\sum_k\frac{\Phi'_k(0)}{k^2}\hat{\alpha}^{(N)}_k\hat{y}_k \qquad \forall\,\hat{y}\in\Dom{J^{(N)}}.
		\end{align*}
		This amounts to satisfying \eqref{WeakEquation for (quasi)} in Definition~\ref{Defn of weak Sol to (quasi)} for functions $\psi(x,t)= \sum_{k\in \Zodd, |k|\leq N} \hat \psi_k(x) e_k(t)$ with $\hat\psi_k\in H^1(\R)$. Clearly, in general $\hat{\alpha}^{(N)}$ is not a critical point of $J$.
	\end{rmk}
	\begin{proof} 
		(i) We choose $N\in\Nodd$ so large, such that we have the assumed sign of the the one element in $\left(\Phi_k'(0)\right)_{\abs{k}\leq N}$. The restriction of $J$ to the $\frac{N+1}{2}$-dimensional space $\Dom{J^{(N)}}$ preserves coercivity.  The continuity of $J^{(N)}$ therefore guarantees the existence of a minimizer $\hat\alpha^{(N)}\in\Dom{J^{(N)}}$. As before we see that $J(\hat\alpha^{(N)})=\inf J^{(N)}<0$, so in particular $\hat{\alpha}^{(N)}\neq0$. Next we observe that $\Dom{J^{(N)}}\subset\Dom{J}$, i.e., $J(\hat{\alpha}^{(N)})\geq\inf J= J(\hat\beta)$ for a minimizer $\hat{\beta}\in\Dom{J}$ of $J$. Let us define $\hat{\beta}^{(N)}_k=\hat{\beta}_k$ for $\abs{k}\leq N$ and $\hat{\beta}^{(N)}_k=0$. Since the Fourier-series $\beta(t) = \sum_k \hat\beta_k e_k(t)$ converges in $L^4(\T)$, cf. Theorem~4.1.8 in \cite{GrafakosClass}, we see that $\hat{\beta}^{(N)}\rightarrow\hat{\beta}$ in $\Dom{J}$. By the minimality of $\hat{\alpha}^{(N)}\in \Dom{J^{(N)}}$ and continuity of $J$ we conclude
		\begin{align*}
			\inf_{\Dom{J}}J\leq J(\hat{\alpha}^{(N)})\leq J(\hat{\beta}^{(N)})\longrightarrow J(\hat{\beta})=\inf_{\Dom{J}}J.
		\end{align*}
		Hence $\lim_{N\to\infty} J(\hat{\alpha}^{(N)})=\inf J$ as claimed.
	\medskip
	
	\noindent
		(ii) Since $\Dom{J^{(N)}}\subset\Dom{J^{(N+1)}}\subset\Dom{J}$ we see that $J(\hat{\alpha}^{(N)})\geq J(\hat{\alpha}^{(N+1)})\geq\inf J$ so that in particular the sequence $(J(\hat{\alpha}^{(N)}))_N$ is bounded. By coercivity of $J$ we conclude that $(\hat{\alpha}^{(N)})_N$ is bounded in $\Dom{J}$ so that there is $\hat{\alpha}\in\Dom{J}$ and a subsequence (again denoted by $(\hat{\alpha}^{(N)})_N$) such that 
		\begin{align*}
		    \hat{\alpha}^{(N)}\rightharpoonup\hat{\alpha} \qquad \text{ in }~\Dom{J}.
		\end{align*}
		By part (i) and weak lower semi-continuity of $J$ we obtain
		\begin{align*}
			\inf J=\lim_{N\to\infty}J(\hat{\alpha}^{(N)})\geq J(\hat{\alpha}),
		\end{align*}
		i.e., $\hat{\alpha}$ is a minimizer of $J$. Recall that $J(\cdot)= \frac{1}{4}\NORM{\cdot}^4+J_1(\cdot)$ where $J_1$ is weakly continuous, cf. proof of Theorem~\ref{J attains a minimum and its properties}. Therefore, since $\hat{\alpha}^{(N)}\rightharpoonup\hat{\alpha}$ and $J(\hat{\alpha}^{(N)})\to J(\hat{\alpha})$ we see that $\NORM{\hat\alpha^{(N)}}\to \NORM{\hat\alpha}$ as $N\to \infty$. Since $\Dom{J}$ is strictly uniformly convex, we obtain the norm-convergence of $(\hat\alpha^{(N)})_N$ to $\hat\alpha$. 
	\end{proof}		

%

\section{Appendix}\label{appendix}

\subsection{Details on exponentially decreasing fundamental solutions for step potentials}\label{details_example_step}
	Here we consider a second-order ordinary differential operator 
	\begin{align*}
		L_k \coloneqq - \frac{d^2}{dx^2} -k^2\omega^2 g(x)
	\end{align*}
	with $g$ as in Theorem~\ref{step}. Clearly, $L_k$ is a self-adjoint operator on $L^2(\R)$ with domain $H^2(\R)$. Moreover, $\sigma_{ess}(L_k)=[k^2\omega^2 a,\infty)$. By the assumption on $\omega$ we have 
	\begin{align*}
		\sqrt{b}\omega c \frac{2}{\pi} = \frac{p}{q} \mbox{ with } p,q \in \Nodd.
	\end{align*}
	Hence, with $k\in q\Nodd$, $k\sqrt{b}\omega c$ is an odd multiple of $\pi/2$. In the following we shall see that $0$ is not an eigenvalue of $L_k$ for $k\in q\Nodd$ so that \eqref{spectralcond} as in Remark~\ref{remark_infinitely} is fulfilled. A potential eigenfunction $\phi_k$ for the eigenvalue $0$ would have to look like
	\begin{equation} \label{ansatz_ef}
    	\phi_k(x) = 
    	\begin{cases}
    		-A\sin(k\omega\sqrt{b}c) e^{k\omega\sqrt{a}(x+c)}, & ~\phantom{-c<}x<-c,\\
    		A\sin(k\omega \sqrt{b} x)+B\cos(k\omega\sqrt{b}x), & -c<x<c,\\
    		A\sin(k\omega\sqrt{b}c) e^{-k\omega\sqrt{a}(x-c)}, & \phantom{-}c<x.
    	\end{cases}
	\end{equation}
    with $A,B\in \R$ to be determined. Note that we have used $\cos(k\omega\sqrt{b}c)=0$. The $C^1$-matching of $\phi_k$ at $x=\pm c$ lead to the two equations
    \begin{align*}
    	-Bk\omega\sqrt{b}\sin(k\omega\sqrt{b}c) &= -Ak\omega\sqrt{a}\sin(k\omega\sqrt{b}c),\\
   		Bk\omega\sqrt{b}\sin(k\omega\sqrt{b}c) &= -Ak\omega\sqrt{a}\sin(k\omega\sqrt{b}c)
    \end{align*}
    and since $\sin(k\omega\sqrt{b}c)=\pm 1$ this implies $A=B=0$ so that there is no eigenvalue $0$ of $L_k$. Next we need to find the fundamental solution $\phi_k$ of $L_k$ that decays to zero at $+\infty$ and is normalized by $\phi_k(0)=1$. Here we can use the same ansatz as in \eqref{ansatz_ef} and just ignore the part of $\phi_k$ on $(-\infty,0)$. Now the normalization $\phi_k(0)=1$ leads to $B=1$ and the $C^1$-matching at $x=c$ leads to $A=\sqrt{\frac{b}{a}}B=\sqrt{\frac{b}{a}}$ so that the decaying fundamental solution is completely determined. We find that 
    \begin{align*}
	    \abs{\phi_k(x)} \leq \left\{ \begin{array}{ll}
	    A+B, & 0\leq x \leq c \vspace{\jot}\\
	    A, & c<x\leq 2c \vspace{\jot}\\
	    A e^{-\frac{1}{2} k\omega\sqrt{a}x}, & x>2c
	    \end{array}\right.
	\end{align*}
    so that $|\phi_k(x)|\leq (A+B)e^{-\rho_k x}\leq Me^{-\rho x}$ on $[0,\infty)$ with $\rho_k = \frac{1}{2} k\omega\sqrt{a}$, $\rho=\frac{1}{2}\omega\sqrt{a}$ and $M=A+B$. This shows that also \eqref{FurtherCond_phik} holds. Finally, since $\phi_k'(0)=\frac{bk\omega}{\sqrt{a}}>0$ the existence of infinitely many breathers can only be shown for $\gamma<0$. At the same time, due to $|\phi_k(0)|=O(k)$, Theorem~\ref{w is even more regular} applies. 
    
\subsection{Details on Bloch Modes for periodic step potentials}
\label{explicit example Bloch Modes_WR}
	Here we consider a second-order periodic ordinary differential operator 
	\begin{align*}
		L \coloneqq - \frac{d^2}{dx^2} + V(x)
	\end{align*}
	with $V\in L^\infty(\R)$ which we assume to be even and $2\pi$-periodic. Moreover, we assume that $0$ does not belong to the spectrum of $L:H^2(\R)\subset L^2(\R)\to L^2(\R)$. We first describe what Bloch modes are and why they exists. Later we show that this is the situation which occurs in Theorem~\ref{w is a weak solution general} and we verify conditions \eqref{spectralcond} and \eqref{FurtherCond_phik}.
	
	\medskip
	
	A function $\Phi\in\Cont{1}{\R}$ which is twice almost everywhere differentiable such that
	\begin{equation} \label{eq:bloch}
		L\Phi=0 \quad\text{ a.e. in } \R, \qquad \Phi(\cdot+2\pi)=\rho\Phi(\cdot). 
	\end{equation}
	with $\rho\in (-1,1)\setminus\{0\}$ is called the (exponentially decreasing for $x\to \infty$) Bloch mode of $L$ and $\rho$ is called the Floquet multiplier. The existence of $\Phi$ is guaranteed by the assumption that $0\notin\sigma(L)$.
	This is essentially Hill's theorem, cf. \cite{Eastham}. Note that $\Psi(x)\coloneqq\Phi(-x)$ is a second Bloch mode of $L$, which is exponentially increasing for $x\to \infty$. The functions $\Phi$ and $\Psi$ form a fundamental system of solutions for operator $L$ on $\R$. Next we explain how $\Phi$ is constructed, why it can be taken real-valued and why it does not vanish at $x=0$ so that we can assume w.l.o.g $\Phi(0)=1$. 
\medskip
	
	According to \cite{Eastham}, Theorem 1.1.1 there are linearly independent functions $\Psi_{1},\Psi_{2}\colon\R\rightarrow\C$ and Floquet-multipliers $\rho_{1},\rho_{2}\in\C$ such that $L\Psi_{j}=0$ a.e. on $\R$ and $\Psi_{j}(x+2\pi)=\rho_{j}\Psi_{j}(x)$ for $j=1,2$. We define $\phi_{j}$, $j=1,2$ as the solutions to the initial value problems
	\begin{align*}
		\begin{cases}
			L\phi_1=0,\\
			\phi_{1}(0)=1,\quad	\phi_{1}'(0)=0,
		\end{cases}
		\quad\text{and}\qquad
		\begin{cases}
			L\phi_2=0,\\
			\phi_{2}(0)=0,\quad	\phi_{2}'(0)=1
		\end{cases}
	\end{align*}
	and consider the Wronskian 
	\begin{align} \label{wronskian}
		W(x)\coloneqq\begin{pmatrix}
			\phi_{1}(x) & \phi_{2}(x) \\
			\phi'_{1}(x) & \phi'_{2}(x)
		\end{pmatrix}
	\end{align}
	and the monodromy matrix 
	\begin{align} \label{monodromy}
		A\coloneqq W(2\pi)=\begin{pmatrix}
			\phi_{1}(2\pi) & \phi_{2}(2\pi) \\
			\phi'_{1}(2\pi) & \phi'_{2}(2\pi)
		\end{pmatrix}.
	\end{align}
	Then $\det A=1$ is the Wronskian determinant of the fundamental system $\phi_1, \phi_2$ and the Floquet multipliers $\rho_{1,2} = \frac{1}{2} \left(\tr(A)\pm \sqrt{\tr(A)^2-4}\right)$ are the eigenvalues of $A$ with corresponding eigenvectors $v_{1}=(v_{1,1}, v_{1,2})\in \C^2$ and $v_{2}=(v_{2,1}, v_{2,2})\in\C^2$. Thus, $\Psi_{j}(x)=v_{j,1}\phi_{1}(x)+v_{j,2}\phi_{2}(x)$. By Hill's theorem (see \cite{Eastham}) we know that
	\begin{align*}
		0\in\sigma(L) \qquad\Leftrightarrow\qquad \abs{\tr(A)}\leq 2.
	\end{align*}
	Due to the assumption that $0\not\in \sigma(L)$ we see that $\rho_1, \rho_2$ are real with $\rho_1, \rho_2\in\R\setminus\{-1,0,1\}$ and $\rho_1\rho_2=1$, i.e., one of the two Floquet multipliers has modulus smaller then one and other one has modulus bigger than one. W.l.o.g. we assume $0<|\rho_2|<1<|\rho_1|$.  Furthermore, since $\rho_1, \rho_2$ are real and $A$ has real entries we can choose $v_1, v_2$ to be real and so $\Psi_1, \Psi_2$ are both real valued. As a result we have found a real-valued Bloch mode $\Psi_2(x)$ which is exponentially decreasing as $x\to \infty$ due to $|\rho_2|<1$. Let us finally verify that $\Psi_2(0)\not = 0$ so that we may assume by rescaling that $\Psi_2(0)=1$. Assume for contradiction that $\Psi_2(0)=0$. Since the potential $V(x)$ is even in $x$ this implies that $\Psi_2$ is odd and hence (due to the exponential decay at $+\infty$) in $L^2(\R)$. But this contradicts that $0\not\in\sigma(L)$. 
\medskip

	Now we explain how the precise choice of the data $a, b>0, \Theta\in (0,1)$ and $\omega$ for the step-potential $g$ in Theorem~\ref{w is a weak solution in expl exa} allows to fulfill the conditions \eqref{spectralcond} and \eqref{FurtherCond_phik}. Let us define
	\begin{align*}
		\tilde g(x)\coloneqq
		\begin{cases}
			a, & x \in [0,2\Theta\pi),\\
			b, & x \in (2\Theta\pi,2\pi).
		\end{cases}
	\end{align*}
	and extend $\tilde g$ as a $2\pi$-periodic function to $\R$. Then $\tilde g(x) = g(x-\Theta\pi)$, and the corresponding exponentially decaying Bloch modes $\tilde\phi_k$ and $\phi_k$ are similarly related by $\tilde\phi_k(x) = \phi_k(x-\Theta\pi)$. For the computation of the exponentially decaying Bloch modes, it is, however, more convenient to use the definition $\tilde g$ instead of $g$. 

	Now we will calculate the monodromy matrix $A_k$ from \eqref{monodromy} for the operator $L_k$.  For a constant value $c>0$ the solution of the initial value problem 
	\begin{align*}
	-\phi''(x)-k^2\omega^2c\phi(x)=0, \quad\phi(x_0)=\alpha,\quad\phi'(x_0)=\beta 
	\end{align*}
	is given by 
	\begin{align*}
		\begin{pmatrix}\phi(x)\\\phi'(x)\end{pmatrix} = 	T_k(x-x_0,c)\begin{pmatrix}\alpha\\\beta\end{pmatrix}
	\end{align*}
	with the propagation matrix
	\begin{align*}
		T_k(s,c)\coloneqq\begin{pmatrix} \cos(k\omega\sqrt{c}s) & \frac{1}{k\omega\sqrt{c}}\sin(k\omega\sqrt{c}s) \\ -k\omega\sqrt{c}\sin(k\omega\sqrt{c}s) & \cos(k\omega\sqrt{c}s) \end{pmatrix}.
	\end{align*}
	Therefore we can write the Wronskian as follows
	\begin{align*}
		W_k(x)
		&=\begin{cases}
			T_k(x,a) & x\in[0,2\Theta\pi] \\
			T_k(x-2\Theta\pi,b)T_k(2\Theta\pi,a) & x\in[2\Theta\pi,2\pi]
		\end{cases}
	\end{align*}
	and the monodromy matrix as 
	\begin{align*}
	    A_k=W_k(2\pi)=T_k(2\pi(1-\Theta),b)T_k(2\Theta\pi,a).
	\end{align*}
	To get the exact form of $A_k$ let us use the notation 
	\begin{align*}
		l\coloneqq\sqrt{\frac{b}{a}}\,\frac{1-\Theta}{\Theta},\qquad  m\coloneqq2\sqrt{a}\Theta\omega.
	\end{align*}
	Hence
	\begin{align*}
		A_k = & \sin(kml\pi)\sin(km\pi) \\
		& \cdot \begin{pmatrix} \cot(kml\pi)\cot(km\pi)-\sqrt{\frac{a}{b}} & \frac{1}{k\omega\sqrt{a}} \cot(kml\pi) + \frac{1}{k\omega\sqrt{b}}\cot(km\pi) \\
		-k\omega\sqrt{b}\cot(km\pi)-k\omega\sqrt{a}\cot(kml\pi) & -\sqrt{\frac{b}{a}}+\cot(kml\pi)\cot(km\pi)
		\end{pmatrix}
	\end{align*}
	and
	\begin{align*}
	\tr(A_k)= 2\cos(kml\pi)\cos(km\pi)-\Bigl(\sqrt{\frac{a}{b}}+\sqrt{\frac{b}{a}}\Bigr)\sin(kml\pi)\sin(km\pi).
	\end{align*}
	In order to verify \eqref{spectralcond} we aim for $\abs{\tr(A_k)}>2$. However, instead of showing $\abs{\tr(A_k)}>2$ for all $k\in\Zodd$ we may restrict to $k\in r\cdot\Zodd$ for fixed $r\in\Nodd$ according to Remark~\ref{remark_Dr}. Next we will choose $r\in\Zodd$. Due to the assumptions from Theorem~\ref{w is a weak solution in expl exa} we have   
	\begin{equation} \label{cond_on_lr}
		l=\frac{\tilde p}{\tilde q}, ~~2m=\frac{p}{q} \in\frac{\Nodd}{\Nodd}.
	\end{equation}
	Therefore, by setting $r=\tilde q q$\footnote{Instead of $r=\tilde q q$ we may have chosen any odd multiple of $\tilde q q$, e.g. $r=(\tilde q q)^j$ for any $j\in \N$. This is important for the applicability of Theorem~\ref{multiplicity abstract} to obtain infinitely many breathers.} we obtain $\cos(km\pi)=\cos(kml\pi)=0$ and $\sin(km\pi), \sin(kml\pi)\in\{\pm1\}$ for all $k\in r\cdot \Zodd$. Together with $a\not =b$ this implies $|\tr(A_k)|=\sqrt{\frac{a}{b}}+\sqrt{\frac{b}{a}}>2$ so that \eqref{spectralcond} holds and $A_k$ takes the simple diagonal form
	\begin{align*}
		A_k = \begin{pmatrix}
		-\sqrt{\frac{a}{b}}\sin(kml\pi)\sin(km\pi) & \\
		0 & -\sqrt{\frac{b}{a}}\sin(kml\pi)\sin(km\pi).
		\end{pmatrix}
	\end{align*}
	In the following we assume w.l.o.g $0<a<b$, i.e., the Floquet exponent with modulus less than $1$ is $\rho_k = -\sqrt{\frac{a}{b}}\sin(kml\pi)\sin(km\pi)$. Note that $|\rho_k|=\sqrt{a/b}$ is independent of $k$. Furthermore the Bloch mode $\tilde \phi_k$ that is decaying to $0$ at $+\infty$ and normalized by $\tilde\phi_k(\Theta\pi)=1$ is deduced from the upper left element of the Wronskian, i.e., 
	\begin{align*}
		\tilde \phi_k(x) = \frac{1}{\cos(k\omega\sqrt{a}\Theta\pi)}\left\{ \begin{array}{ll}
		\cos(k\omega\sqrt{a}x), & x\in (0,2\Theta\pi), \vspace{\jot}\\
		\cos(k\omega\sqrt{b}(x-2\Theta\pi))\cos(k\omega\sqrt{a}2\Theta\pi) & \\
		-\sqrt{\frac{a}{b}}\sin(k\omega\sqrt{b}(x-2\Theta\pi))\sin(k\omega\sqrt{a}2\Theta\pi), & x\in (2\Theta\pi,2\pi)
		\end{array}
		\right.
	\end{align*}
	and on shifted intervals of lengths $2\pi$ one has $\tilde\phi_k(x+2m\pi)= \rho_k^{m}\tilde\phi_k(x)$. Notice that by \eqref{cond_on_lr} the expression $k\omega\sqrt{a}\Theta\pi=k\frac{p}{q}\frac{\pi}{4}$ is an odd multiple of $\pi/4$ since $k\in q\tilde q\Zodd$ and hence $|\cos(k\omega\sqrt{a}\Theta\pi)|=1/\sqrt{2}$. Therefore $\|\phi_k\|_{L^\infty(0,\infty)}=\|\tilde\phi_k\|_{L^\infty(\Theta\pi,\infty)}\leq \|\tilde\phi_k\|_{L^\infty(0,2\pi)}\leq \sqrt{2}(1+\sqrt{a/b})$. Thus we have shown that $|\phi_k(x)|\leq M e^{-\rho x}$ for $x\in [0,\infty)$ with $M>0$ and $\rho=\frac{1}{4\pi}(\ln b-\ln a)>0$. Finally, let us compute
	\begin{align*}
		\phi_k'(0)=\tilde\phi_k'(\Theta\pi)= -k\omega\sqrt{a}\tan(k\omega\sqrt{a}\Theta\pi)\in\{\pm k\omega\sqrt{a}\}.
	\end{align*}
	This shows that $|\phi_k'(0)|=O(k)$ holds which allows to apply Theorem~\ref{w is even more regular}. It also shows that the estimate $|\phi_k(0)|=O(k^\frac{3}{2})$ from Lemma~\ref{norm_estimates} can be improved in special cases. To see that $\phi_k'(0)$ is alternating in $k$, observe that moving from $k\in r\Zodd$ to $k+2r\in r\Zodd$ the argument of $\tan$ changes by $2r\omega\sqrt{a}\Theta\pi$ which is an odd multiple of $\pi/2$. Since $\tan(x+\Zodd\frac{\pi}{2})=-1/\tan(x)$ we see that the sequence $\phi_k'(0)$ is alternating for $k\in r\Zodd$. This shows in particular that for any $j\in\N$ the sequence $(\phi_{hr^j}'(0))_{h\in \Nodd}$ contains infinitely many positive and negative elements, and hence Theorem~\ref{multiplicity abstract} for the existence of infinitely many breathers is applicable. This concludes the proof Theorem~\ref{w is a weak solution in expl exa} since we have shown that the potential $g$ satisfies the assumptions \eqref{spectralcond} and \eqref{FurtherCond_phik} from Theorem~\ref{w is a weak solution general}. 

\subsection{Embedding of H\"older-spaces into Sobolev-spaces}\label{embedding}

	\begin{lemma} \label{unusual_embedd} 
		For $0<\tilde\nu<\nu<1$ there is the continuous embedding $\Cont{0,\nu}{\T_T}\to \SobH{\tilde\nu}{\T_T}$. 
	\end{lemma}
	\begin{proof} 
		Let $z(t)= \sum_k \hat z_k e_k(t)$ be a function in $\Cont{0,\nu}{\T_T}$. We need to show the finiteness of the spectral norm $\|z\|_{H^{\tilde\nu}}$. For this we use the equivalence of the spectral  norm $\|\cdot\|_{H^{\tilde\nu}}$ with the Slobodeckij norm, cf. Lemma~\ref{equivalence}. Therefore it suffices to check the estimate
		\begin{align*}
			\int_{\T_T} \int_{\T_T} \frac{|z(t)-z(\tau)|^2}{|t-\tau|^{1+2\tilde\nu}} \dd{t}\dd{\tau} \leq 	\|z\|_{C^\nu(\T_T)}^2 \int_{\T_T} \int_{\T_T} |t-\tau|^{-1+2(\nu-\tilde\nu)} \dd{t}\dd{\tau} \leq C(\nu,\tilde\nu)\|z\|_{C^\nu(\T_T)}^2
		\end{align*}
		where the double integral is finite due to $\nu>\tilde\nu$. 
	\end{proof}
	
	For $0<s<1$ recall the definition of the Slobodeckij-seminorm for a function $z:\T_T \to \R$
	\begin{align*}
		[z]_s \coloneqq \left(\int_{\T_T} \int_{\T_T} \frac{|z(t)-z(\tau)|^2}{|t-\tau|^{1+2s}} 	\dd{t}\dd{\tau}\right)^{1/2}.
	\end{align*}
	
	\begin{lemma} \label{equivalence} 
		For functions $z\in H^s(\T_T)$, $0<s<1$ the spectral norm $\|z\|_{H^s} = (\sum_k (1+k^2)^s |\hat z_k|^2)^{1/2}$ and the Solobodeckij norm $\NORM{z}_{H^s}\coloneqq (\|z\|_{L^2(\T_T)}^2+ [z]_s^2)^{1/2}$ are equivalent.
	\end{lemma}
	
	\begin{proof} 
		The Solobodeckij space and the spectrally defined fractional Sobolev space are both Hilbert spaces. Hence, by the open mapping theorem, if suffices to verify the estimate $\NORM{z}_{H^s} \leq C\|z\|_{H^s}$. By direct computation we get
		\begin{align*}
			\int_{\T_T} \int_{\T_T} \frac{|z(t)-z(\tau)|^2}{|t-\tau|^{1+2s}}\dd{t}\dd{\tau} &= \int_0^{T} 	\int_{-\tau}^{T-\tau} \frac{|z(x+\tau)-z(\tau)|^2}{|x|^{1+2s}} \dd{x}\dd{\tau}\\
			&= \int_0^{T} \left(\int_0^{T-\tau} \frac{|z(x+\tau)-z(\tau)|^2}{x^{1+2s}}\dd{x} + 	\int_{T-\tau}^{T} \frac{|z(x+\tau)-z(\tau)|^2}{(T-x)^{1+2s}}\dd{x}\right)\dd{\tau} \\
			&= \int_0^{T}\int_0^{T} \frac{|z(x+\tau)-z(\tau)|^2}{g(x,\tau)^{1+2s}}\dd{x}\dd{\tau}
		\end{align*}
		with 
		\begin{align*}
			g(x,\tau) = \left\{
			\begin{array}{ll} 
			x & \mbox{ if } 0\leq x\leq T-\tau, \vspace{\jot}\\
			T-x & \mbox{ if } T-\tau \leq x \leq T.
			\end{array}
			\right.
		\end{align*}
		Since $g(x,\tau) \geq \dist(x,\partial\T_T)$ and due to Parseval's identity we find 
		\begin{align*}
			\int_{\T_T} \int_{\T_T} \frac{|z(t)-z(\tau)|^2}{|t-\tau|^{1+2s}}\dd{t}\dd{\tau} & \leq 	\int_{\T_T} \frac{\| \widehat{z(\cdot+x)}-\hat z\|_{l^2}^2}{\dist(x,\partial\T_T)^{1+2s}}\dd{x} \\
			&= \int_{\T_T} \sum_k \frac{|\exp(\i k\omega x)-1|^2 |\hat 	z_k|^2}{\dist(x,\partial\T_T)^{1+2s}}\dd{x} \\
			&= 4 \int_0^{T/2}\sum_k \frac{1-\cos(k\omega x)}{x^{1+2s}} |\hat z_k|^2 \dd{x} \\
			& \leq 4\tilde C \sum_k k^{2s} |\hat z_k|^2
		\end{align*}
		with $\tilde C=\int_0^\infty \frac{1-\cos(\omega\xi)}{\xi^{1+2s}}\dd{\xi}$. This finishes the proof.
	\end{proof}

\section*{Acknowledgment}
	Funded by the Deutsche Forschungsgemeinschaft (DFG, German Research Foundation) – Project-ID 258734477 – SFB 1173

\bibliographystyle{plain}
\bibliography{bibliography}

\end{document}